 \documentclass[11pt,a4paper,reqno]{amsart}
\usepackage[utf8]{inputenc}
\usepackage{amsmath, amsthm}
\usepackage{amsfonts}
\usepackage{amssymb}
\usepackage{graphicx}
\usepackage[left=2.5cm,right=2.5cm,top=2cm,bottom=2cm]{geometry}
\usepackage{amscd}
\usepackage[all]{xy}
\usepackage{hyperref}
\usepackage{gensymb}
\usepackage{enumerate}
\usepackage{csquotes}
\usepackage{comment}
\usepackage{hhline}
\usepackage{float}
\usepackage{subfigure}
\usepackage{color}

\usepackage{caption} 
\theoremstyle{plain}
\newtheorem{theorem}{Theorem}
\newtheorem{proposition}[theorem]{Proposition}
\newtheorem{lemma}[theorem]{Lemma}

\theoremstyle{definition}
\newtheorem{example}{Example}
\newtheorem{definition}{Definition}

\theoremstyle{remark}
\newtheorem*{remark}{Remark}

\newcommand{\ignore}[1]{}

\begin{document}

\title[Endemic equilibria and control via vaccination and mitigation] 
      {A look at endemic equilibria of compartmental epidemiological models and model control via vaccination and mitigation}

\author{Monique Chyba}
\address{University of Hawai`i at M\={a}noa, 2565 McCarthy Mall,  Honolulu, HI 96822, USA}
\email{chyba@hawaii.edu}

\author{Taylor Klotz}
\address{University of Hawai`i at M\={a}noa, 2565 McCarthy Mall,  Honolulu, HI 96822, USA}
\email{klotz@hawaii.edu}

\author{Yuriy Mileyko}
\address{University of Hawai`i at M\={a}noa, 2565 McCarthy Mall,  Honolulu, HI 96822, USA}
\email{ymileyko@hawaii.edu}

\author{Corey Shanbrom}
\address{California State University, Sacramento, 6000 J St., Sacramento, CA 95819, USA}
\email{corey.shanbrom@csus.edu}



\begin{abstract}
Compartmental models have long served as important tools in mathematical epidemiology, with their usefulness highlighted by the recent COVID-19 pandemic.  However, most of the classical models fail to account for certain features of this disease and others like it, such as the ability of exposed individuals to recover without becoming infectious, or the possibility that asymptomatic individuals can indeed transmit the disease but at a lesser rate than the symptomatic. 

In the first part of this paper we propose two new compartmental epidemiological models and study their equilibria, obtaining an endemic threshold theorem for the first model.  In the second part of the paper, we treat the second model as an affine control system with two controls: vaccination and mitigation.  We show that this system is static feedback linearizable, present some simulations, and investigate of an optimal control version of the problem.  We conclude with some open problems and ideas for future research.
\end{abstract}

\keywords{Compartmental model, endemic equilibrium, endemic threshold, static feedback linearization, optimal control}



\maketitle

\section{Introduction} \label{sec: intro}

When modeling epidemics, compartmental models are vital for studying infectious diseases by providing a way to analyze the dynamics of the disease spread over time. The most basic of such models is known as the SIR model, which groups the population into three compartments (susceptible, infectious, recovered) and has a simple flow where an individual moves from susceptible to infectious to recovered (\cite{Bjornstad1}). An additional compartment called the exposed group, can be included to obtain the SEIR model (\cite{Martcheva}). This model is better suited for diseases with a latent period, the time when an individual has contracted the disease but is unable to infect others.  

Although the SEIR model is a more accurate portrayal of an infectious disease than the SIR model, as most infectious diseases have a latent period, one limitation of this model is that it assumes that when someone recovers from the disease, they are immune to it forever. This is unrealistic because one can lose immunity over time.  
   The SEIRS model is used to rectify this issue as this model assumes that individuals in the recovered group are able to return to the susceptible group (\cite{Bjornstad2}). 

The COVID-19 pandemic highlighted a few shortcomings of the classical SEIR models, as the latter do not account for the possibility of individuals staying asymptomatic throughout the course of the infection (\cite{Caturano, Johansson}) but at the same time capable of infecting susceptible individuals (\cite{Gandhi, Nogrady, Pollock}). One of the first models that looked into such a possibility was due to E. Sontag and collaborators in \cite{Sontag1}, where the authors studied the effects of social distancing.
Here we generalize the standard SEIR model in a similar way by introducing the SE(R)IRS model (Section \ref{sec: new SE(R)IRS}). As the acronym suggests, individuals in the $E$ compartment can pass directly to the $R$ compartment without ever entering the $I$ compartment. We should point out that in this paper we think of the $E$ compartment as representing infected but asymptomatic individuals, while the $I$ compartment represents individuals who are both infected and symptomatic. We have chosen to retain the traditional labeling for simplicity, although some modern authors use the letter $A$ to denote the asymptomatic compartment \cite{Sontag1}.
We also allow individuals from the $E$ compartment to infect those who are susceptible, although at a lesser rate than those from the $I$ compartment.

In Section \ref{sec: new SVE(R)IRS} we generalize further, adding a $V$ compartment representing vaccinated individuals.  The resulting model, which we call SVE(R)IRS, is significantly more complicated.  We are still able to characterize equilibria in terms of the basic reproductive number, but fail to prove an endemic threshold theorem; thus this section is an open invitation to future research.  

Creating a model is one thing, while analyzing it is quite another.  In the first half of the paper we have chosen to focus our analysis on the equilibria of the systems and their stability.  This choice was also motivated by COVID-19 and what will be the ``end" of the pandemic.  As the concept of herd immunity has received much attention in both the media and academia (\cite{Anderson, Aschwanden, Britton, Fontanet, Randolph}), the same cannot be said about of the notion of endemic equilibrium. There is a mathematical foundation for the idea of herd immunity (\cite{Hethcote}), but as we demonstrate in Section \ref{sec: discussion}, this does not mean the disease is eradicated.  It is compatible with what we consider the more relevant idea of an endemic equilibrium: that the disease will always exist (hopefully in small enough numbers to no longer characterize a pandemic).  Moreover, the stability of such an endemic equilibrium would reflect the possibility that new outbreaks or variants could cause spikes in infections, but that over time these numbers would drift back towards some state of ``new normal" (see \cite{Florin3} for a nice overview of stability for SIRS models). The main idea is to design maintenance strategies to control the dynamic of the spread of the disease (through a yearly vaccine or seasonal non-pharmaceutical measures) to stay in a neighborhood of a sustainable endemic equilibrium.

While this discussion makes clear that our models and objects of study are motivated by COVID-19, we hope that our contributions can be applied to other infectious diseases with similar characteristics, including those yet to be discovered.  

As in any mathematical modeling, there is naturally a trade-off between a model's complexity and its accuracy.  In many ways the classical SIR model is useful mainly due to its simplicity, making both mathematical analysis and simulations painless.  But its accuracy may be consequently limited.  On the other hand, much more complicated compartmental models, such as that in \cite{Chyba}, may represent the dynamics of the disease very well at the expense of being computationally difficult.  We hope that, like the popular SEIRS model, the models introduced here strike a reasonable balance by being simple enough for elementary dynamical systems theory and computations, while proving more flexible and accurate than the SEIRS model.  In particular, Theorem \ref{thm: main thm} characterizes the stability of both the endemic and disease-free equilibria in terms of the basic reproductive number $\mathfrak R_0$ using only basic theory.  But ignoring the effects of vaccination greatly misrepresents the course of pandemics like COVID-19.  Yet our SVE(R)IRS model, while more accurate, was just complicated enough that similar analyses failed and we could prove no such theorem.  For these reasons we believe these models lie somewhere near the right balance of complexity and simplicity.  To our knowledge, neither has appeared in the literature before.

The second half of the paper, Section \ref{sec: control}, treats the SVE(R)IRS model as an affine control system with two controls: vaccination and mitigation. This approach is similar to that of E. Sontag and collaborators in such works as \cite{Sontag4, Sontag2, Sontag3}. However, while the authors of the latter paper focus on a single input control system, with the control representing non-pharmaceutical mitigation measures, we consider a bi-input control system, with the second control representing a vaccination rate.  We study the resulting control problem from several perspectives.  In Section \ref{subsec: SFL} we prove that the system is static feedback linearizable, giving the explicit feedback transformation, and observe that it maps equilibria to equilibria; we also analyze the boundary conditions of both the original and transformed systems.  In Section \ref{subsec: simulations} we choose conceptually realistic control curves and explore the corresponding trajectories in both the original and linearized systems.  In Section \ref{subsec: optimal} we fix one control, and consider the 1-input system from an optimal control perspective.  For the time-minimal problem we employ the Maximum Principle and analyze the singular controls.

We should mention that applications of control systems in biological and medical fields has seen an immense contribution from E. Sontag. In addition to producing a plethora of very influential papers, he trained and mentored generations of talented scientists and mathematicians who are now continuing in his footsteps and work on further expanding the applicability of control systems in various scientific disciplines.

\section{Two new compartmental models}
\subsection{SE(R)IRS model}\label{sec: new SE(R)IRS}

The usual SEIRS model (\cite{Bjornstad2}) can be visualized as
$$ \xymatrix{
S  \ar[r]^{\beta I/n} &   E\ar[r]^{\sigma} &  I\ar[r]^{\gamma} & R \ar@/^1pc/[lll]^{\omega}
}
$$
where
\begin{itemize}
    \item $\beta$ is the transmission rate, the average rate at which an infected individual can infect a susceptible
    \item $n$ is the population size
    \item $1/\sigma$ is the latency period
    \item $1/\gamma$ is the symptomatic period
    \item $1/\omega$ is the period of immunity.
\end{itemize}
All parameters are necessarily non-negative.
Note that, for simplicity, here and throughout this paper we choose to present our models without vital dynamics (also known as demography), often represented by the natural birth and death rates $\Lambda$ and $\mu$. This is motivated by the fact that compartmental models are well suited for larger population since they average populations into compartments, and therefore  the change in total population due to overall death and birth is negligible. See \cite{Chengjun} for an analysis with varying total population on a slightly different model (similar results). Also note that when $\omega=0$, this reduces to the usual SEIR model, and one can further recover the simple SIR model by letting $\sigma \to \infty$.   

As described in Section \ref{sec: intro}, we now think of individuals in the $E$ compartment as infected but asymptomatic (some authors call this interpretation a SAIRS model), while individuals in the $I$ compartment are infected and symptomatic.  Therefore in our SE(R)IRS model certain asymptomatic individuals can recover without ever becoming symptomatic; the duration of the course of their infection is denoted by $1/\delta$.  Moreover, asymptomatic individuals can indeed infect susceptible individuals, however they do so at a reduced rate when compared to symptomatic individuals; this reduction is accounted for by the parameter $\alpha$.  We may assume $\alpha \in [0,1]$ and $\delta\geq 0$.  Note that when $\alpha =\delta=0$ we recover the SEIRS model.

The SE(R)IRS model can be visualized as
$$
\xymatrix{
S  \ar[rr]^{\beta (I+\alpha E)/n} &  & E\ar[r]^{\sigma} \ar@/^2pc/[rr]^{\delta} &  I\ar[r]^{\gamma} & R \ar@/^1pc/[llll]^{\omega}
}
$$
and the corresponding dynamical system is given by 
\begin{align}
\frac{d S}{dt} &= -\beta S (I+\alpha E)/n + \omega R \label{eq-1}\\
\frac{d E}{dt} &=  \beta S (I+\alpha E)/n - (\sigma + \delta) E\\
\frac{d I}{dt} &= \sigma E - \gamma I \\
\frac{d R}{dt} &= \delta E + \gamma I - \omega R.\label{eq-4}
\end{align}

One of the key concepts in epidemiology is the basic reproduction number, $\mathfrak R_0$. It is defined as the number of individuals that are infected by a single infected individual during its entire course of infection, in an entire susceptible population.
\begin{proposition}
The SE(R)IRS basic reproductive number is 
$$ \mathfrak{R}_0 = \left(\frac{\alpha \gamma +\sigma}{\delta+\sigma}\right) \left(\frac{\beta}{\gamma} \right).
$$
\end{proposition}

\begin{proof}
We follow \cite{Heffernan} by computing $\mathfrak R_0$ as the spectral radius of the next generation matrix.  First, it is easy to see that the system has a disease-free equilibrium at $(S,E,I,R)=(n, 0, 0, 0)$.
We compute 
$$
F=\begin{pmatrix}
\alpha \beta & \beta \\ 0 & 0 
\end{pmatrix}
\qquad \text{and} \qquad 
V=\begin{pmatrix}
\sigma + \delta & 0 \\ -\sigma & \gamma 
\end{pmatrix};
$$
see \cite{Heffernan} for their precise definitions (matrices of partial derivatives of the rate of appearance of new infections and of the rate of transfer of individuals between compartments, respectively). Thus our next generation matrix is 
$$
FV^{-1}=\begin{pmatrix}
\frac{\beta (\alpha \gamma +\sigma)}{\gamma(\delta+\sigma)} & \frac{\beta }{\gamma} \\ 0 & 0 
\end{pmatrix}.
$$
The basic reproductive number $\mathfrak R_0$ is the spectral radius of this operator, which is the largest eigenvalue $ \left(\frac{\alpha \gamma + \sigma}{\delta +\sigma }\right) \left(\frac{\beta}{\gamma} \right).$
\end{proof}

Note that equations (\ref{eq-1})-(\ref{eq-4}) imply that $n=S+E+I+R$ is constant, an assumption which is common in mathematical epidemiology (\cite{Martcheva}).
This allows us to reduce to a 3-dimensional system in $S,E,I$, with the reduced equations given by
\begin{align}
\frac{d S}{dt} &= -\beta S (I+\alpha E)/n + \omega (n-S-E-I) \label{eq-5}\\
\frac{d E}{dt} &=  \beta S (I+\alpha E)/n - (\sigma + \delta) E\\
\frac{d I}{dt} &= \sigma E - \gamma I.\label{eq-7}
\end{align}
In the sequel we will study this simpler version of the system, recovering the value of $R$ when convenient.

Our first result, proved in the next two sections, is the following.
\begin{theorem} \label{thm: main thm}
If $ \mathfrak R_0<1$ then the disease-free equilibrium is locally asymptotically stable and the endemic equilibrium is irrelevant.  
If $ \mathfrak R_0>1$ then the endemic equilibrium is locally asymptotically stable and the disease-free equilibrium is unstable.
\end{theorem}

\begin{proof}
The theorem follows from Lemmas \ref{lemma: new endemic relevance}, and \ref{lemma: new endemic stability}, and \ref{lemma: new disease free}.
\end{proof}

Here ``irrelevant" means epidemiologically nonsensical, as certain compartments would contain negative numbers of people; it still exists mathematically.
This theorem is sometimes known as the ``endemic threshold property", where $\mathfrak R_0$ is considered a critical threshold.  In \cite{Hethcote}, Hethcote describes this property as ``the usual behavior for an endemic model, in the sense that the disease dies out below the threshold, and the disease goes to a unique endemic equilibrium above the threshold."  The SEIR version is derived nicely in Section 7.2 of \cite{Martcheva}.  
The SEIRS version can be found in \cite{Liu}.

\subsubsection{Analysis of endemic equilibria}
Our system (\ref{eq-5})-(\ref{eq-7}) has a unique endemic equilibrium at 
$$p =(S,E,I)= \frac{n}{\mathfrak R_0}\left(1,\, \omega \epsilon, \, \frac{\sigma \omega}{\gamma} \epsilon \right)$$
where
\begin{equation}
\epsilon =   \frac{1}{\delta + \sigma}\ \frac{\beta(\alpha \gamma + \sigma) - \gamma (\delta + \sigma)}{\sigma \omega + \gamma ( \delta + \sigma + \omega)}.
\end{equation}
If desired, one can determine from the constant total population the recovered population at this equilibrium: $R= (\sigma + \delta)\epsilon$.

Note that this endemic equilibrium  is only realistic if all coordinates are positive, which requires $\epsilon$ positive, 
which is equivalent to
\begin{equation}
    \mathfrak{R}_0>1.
\end{equation}
In other words, we have the following.
\begin{lemma}\label{lemma: new endemic relevance}
 If $\mathfrak R_0 >1$ then there exists a unique endemic equilibrium for the SE(R)IRS model.  If $\mathfrak R_0<1$ then no endemic equilbrium exists.
\end{lemma}

The linearization of the reduced system at $p$ is 
$$L
=\begin{pmatrix}
-\epsilon \omega (\delta + \sigma)- \omega & -\frac{\alpha\beta}{\mathfrak R_0}  -\omega & -\frac{\beta}{\mathfrak R_0} - \omega \\
\epsilon \omega (\delta + \sigma) & -\frac{\beta \sigma }{\gamma \mathfrak R_0} & \frac{\beta}{\mathfrak R_0} \\
0 & \sigma & -\gamma 
\end{pmatrix}.
$$
As expected, this matrix does not depend on $n$. 
Unfortunately, the eigenvalues of $L$ are not analytically computable for general parameters.  As such, we implement a different criteria from \cite{Fuller} in the proof of Lemma \ref{lemma: new endemic stability}, which allows one to determine stability without explicitly computing all eigenvalues.
Note that $L$ is nonsingular for generic parameter values.  But the determinant does indeed vanish if and only if $\mathfrak R_0 =1.$

\begin{lemma}\label{lemma: new endemic stability}
If $\mathfrak R_0>1$ then the endemic equilibrium is locally asymptotically stable.
\end{lemma}

\begin{proof}
We apply the criteria (12.21-12.23) from \cite{Fuller} to $L$: if the determinant and trace of $L$, as well as the determinant of the bialternate product of $L$ with itself, are all negative then all eigenvalues have negative real parts.  Note that the trace is obviously negative, so (12.22) is immediately satisfied.  Now compute
$$\det(L)=-\epsilon \omega (\delta + \sigma) (\sigma \delta + \gamma (\delta + \sigma + \omega))$$
which is clearly negative, so (12.21) is satisfied.  
Finally,  we compute the bialternate sum of $L$ with itself (see \cite{Ludwig} or \cite{Fuller} for a precise definition), 
$$ G= \begin{pmatrix}
-\epsilon \omega (\delta + \sigma) - \omega - \frac{\beta \sigma}{\gamma \mathfrak R_0}   &  \frac{\beta}{\mathfrak R_0}   & \frac{\beta}{\mathfrak R_0} +\omega \\
\sigma &  -\epsilon \omega (\delta + \sigma) - \omega - \gamma & -\frac{\alpha \beta}{\mathfrak R_0} -\omega \\
0 & \epsilon \omega (\delta + \sigma)  &  -\frac{\beta \sigma }{\gamma \mathfrak R_0} -\gamma
\end{pmatrix}
$$
whose determinant
\begin{align*}
\det G &= -\frac{\omega}{(\alpha \gamma + \sigma)^2}
 (\alpha^2 \gamma^4 + \alpha^2 \gamma^4 \delta \epsilon + 2 \alpha \gamma^3 \sigma + 2 \alpha \gamma^2 \delta \sigma + \alpha^2 \gamma^4 \epsilon \sigma + 
    2 \alpha \gamma^3 \delta \epsilon \sigma \\
    &+ 
    \alpha \gamma^2 \delta^2 \epsilon \sigma + \alpha \gamma \delta^3 \epsilon \sigma + \gamma^2 \sigma^2 + 
    2 \alpha \gamma^2 \sigma^2 + 2 \gamma \delta \sigma^2 + \delta^2 \sigma^2 + 2 \alpha \gamma^3 \epsilon \sigma^2 + \gamma^2 \delta \epsilon \sigma^2 \\
    &+ 
    2 \alpha \gamma^2 \delta \epsilon \sigma^2 + \gamma \delta^2 \epsilon \sigma^2 + 3 \alpha \gamma \delta^2 \epsilon \sigma^2 + \delta^3 \epsilon \sigma^2 + 
    2 \gamma \sigma^3 + 2 \delta \sigma^3 + \gamma^2 \epsilon \sigma^3 + \alpha \gamma^2 \epsilon \sigma^3 \\
    &+
    2 \gamma \delta \epsilon \sigma^3 + 
    3 \alpha \gamma \delta \epsilon \sigma^3 + 3 \delta^2 \epsilon \sigma^3 + \sigma^4 + \gamma \epsilon \sigma^4 + \alpha \gamma \epsilon \sigma^4 + 
    3 \delta \epsilon \sigma^4 + \epsilon \sigma^5 + \alpha^2 \gamma^3 \omega \\
    &+ 
    2 \alpha^2 \gamma^3 \delta \epsilon \omega    + 
    \alpha^2 \gamma^2 \delta^2 \epsilon \omega + \alpha^2 \gamma^3 \delta^2 \epsilon^2 \omega + \alpha^2 \gamma^2 \delta^3 \epsilon^2 \omega + 
    2 \alpha \gamma^2 \sigma \omega + \alpha \gamma \delta \sigma \omega \\
    &+ 2 \alpha^2 \gamma^3 \epsilon \sigma \omega + 4 \alpha \gamma^2 \delta \epsilon \sigma \omega + 
    \alpha^2 \gamma^2 \delta \epsilon \sigma \omega + 4 \alpha \gamma \delta^2 \epsilon \sigma \omega + 2 \alpha^2 \gamma^3 \delta \epsilon^2 \sigma \omega \\
    &+ 
    2 \alpha \gamma^2 \delta^2 \epsilon^2 \sigma \omega + 3 \alpha^2 \gamma^2 \delta^2 \epsilon^2 \sigma \omega + 2 \alpha \gamma \delta^3 \epsilon^2 \sigma \omega + 
    \gamma \sigma^2 \omega + \alpha \gamma \sigma^2 \omega + \delta \sigma^2 \omega \\
    &+
    4 \alpha \gamma^2 \epsilon \sigma^2 \omega + 2 \gamma \delta \epsilon \sigma^2 \omega + 
    6 \alpha \gamma \delta \epsilon \sigma^2 \omega + 3 \delta^2 \epsilon \sigma^2 \omega + \alpha^2 \gamma^3 \epsilon^2 \sigma^2 \omega + 
    4 \alpha \gamma^2 \delta \epsilon^2 \sigma^2 \omega \\
    &+
    3 \alpha^2 \gamma^2 \delta \epsilon^2 \sigma^2 \omega + \gamma \delta^2 \epsilon^2 \sigma^2 \omega + 
    6 \alpha \gamma \delta^2 \epsilon^2 \sigma^2 \omega + \delta^3 \epsilon^2 \sigma^2 \omega + \sigma^3 \omega + 2 \gamma \epsilon \sigma^3 \omega \\ 
    &+ 
    2 \alpha \gamma \epsilon \sigma^3 \omega + 5 \delta \epsilon \sigma^3 \omega + 2 \alpha \gamma^2 \epsilon^2 \sigma^3 \omega + 
    \alpha^2 \gamma^2 \epsilon^2 \sigma^3 \omega + 2 \gamma \delta \epsilon^2 \sigma^3 \omega + 6 \alpha \gamma \delta \epsilon^2 \sigma^3 \omega \\
    &+ 
    3 \delta^2 \epsilon^2 \sigma^3 \omega + 2 \epsilon \sigma^4 \omega + \gamma \epsilon^2 \sigma^4 \omega + 2 \alpha \gamma \epsilon^2 \sigma^4 \omega + 
    3 \delta \epsilon^2 \sigma^4 \omega + \epsilon^2 \sigma^5 \omega + \alpha^2 \gamma^2 \delta \epsilon \omega^2 \\
    &+ 
    \alpha^2 \gamma^2 \delta^2 \epsilon^2 \omega^2 + \alpha^2 \gamma^2 \epsilon \sigma \omega^2 + 2 \alpha \gamma \delta \epsilon \sigma \omega^2 + 
    2 \alpha^2 \gamma^2 \delta \epsilon^2 \sigma \omega^2 + 2 \alpha \gamma \delta^2 \epsilon^2 \sigma \omega^2 \\
    &+
    2 \alpha \gamma \epsilon \sigma^2 \omega^2 + 
    \delta \epsilon \sigma^2 \omega^2 + \alpha^2 \gamma^2 \epsilon^2 \sigma^2 \omega^2 + 4 \alpha \gamma \delta \epsilon^2 \sigma^2 \omega^2 + 
    \delta^2 \epsilon^2 \sigma^2 \omega^2 + \epsilon \sigma^3 \omega^2 \\
    &+
    2 \alpha \gamma \epsilon^2 \sigma^3 \omega^2 + 
    2 \delta \epsilon^2 \sigma^3 \omega^2 + \epsilon^2 \sigma^4 \omega^2)
\end{align*}
is also negative, satisfying (12.23).
\end{proof}

\begin{example}\label{ex: new endemic}
In all of our examples the time units are taken to be days, and we choose the parameter values
$$(\alpha, \gamma, \delta, n, \sigma, \omega)=
\Big(\frac{1}{10},  \frac{1}{7}, \frac{1}{14}, 100, \frac{1}{7}, \frac{1}{90} \Big).$$
These numbers are inspired by some of the author's work experience with the State of Hawai`i's COVID-19 response. A helpful reference for this literature is \cite{Carney}. While precise values are not known, and many of these parameters depend on the specific virus variant, these are at least roughly in agreement with some of the literature.  Specifically, we assume that an asymptomatic individual is 10\% as infectious as a symptomatic one, that individuals who become symptomatic have seven day periods of latency and of symptoms, that individuals who never develop symptoms are infected for 14 days, and that the period of immunity is 90 days.  We choose the population size of 100 simply so that compartment values can be interpreted as percentages of a generic population.

When $\beta = 0.4$ we have 
$$ \mathfrak R_0 \approx 2.053$$
and
$$p=(S,E,I,R)\approx(49,  2, 2, 46).$$
The eigenvalues of $L$ are 
$$\lambda_1 \approx -.340, \quad  \lambda_2 \approx -.010 - .031i, \quad  \lambda_3 \approx -.010 + .031i.$$
These three eigenvalues all have negative real part, so the equilibrium is stable.  It appears to be a spiral sink, signifying epidemic waves (\cite{Bjornstad2}), as shown in Figure \ref{fig: endemic}.  

\end{example}

\begin{figure}[H]
\includegraphics[width=7.0cm]{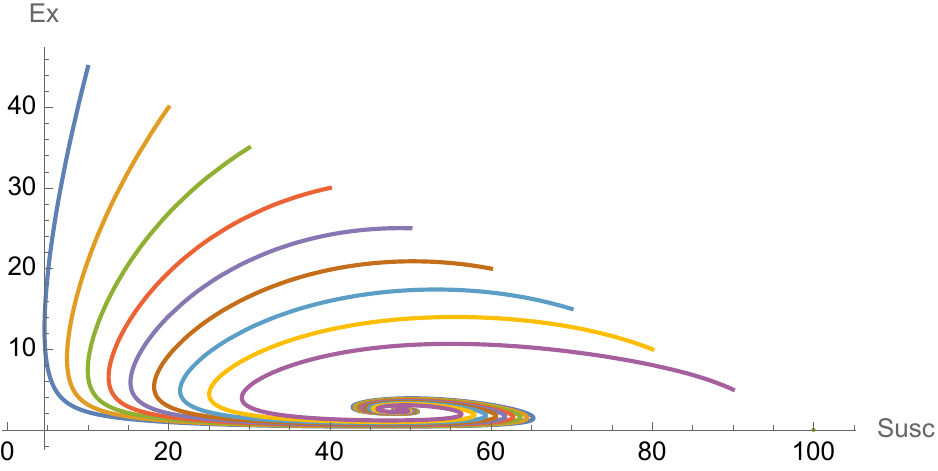}
\includegraphics[width=5.5cm]{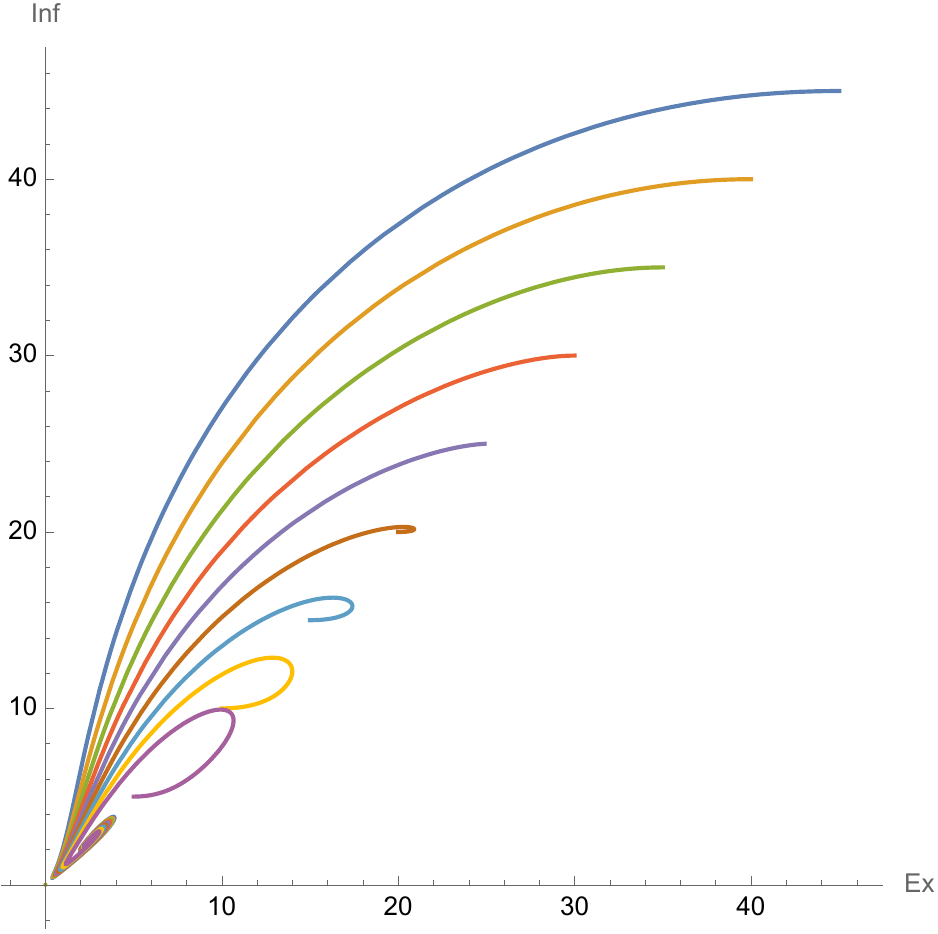}
\caption{The endemic equilibrium of Example \ref{ex: new endemic}. The left plot displays $E$ versus $S$ compartments, while the right plot displays $I$ versus $E$ compartments. Each colored curve represents a different initial condition of the form $(S,E,I,R)=(S_0, (n-S_0)/2, (n-S_0)/2, 0)$ for $S_0 = 10, 20, \dots, 90$.}\label{fig: endemic}
\end{figure}



\subsubsection{Analysis of disease-free equilibria}
One easily checks that the system has a disease-free equilibrium at $(S,E,I,R)=(n, 0, 0, 0)$.
The linearization of the reduced system there is 
$$N=
\begin{pmatrix}
 - \omega & -\alpha\beta-\omega & -\beta - \omega \\
0& \alpha\beta-\delta-\sigma & \beta \\
0 & \sigma & -\gamma 
\end{pmatrix}.
$$
This is singular if and only if $\mathfrak R_0=1$,  just like $L$.

\begin{lemma}\label{lemma: new disease free}
The disease-free equilibrium is locally asymptotically  stable if and only if $\mathfrak R_0<1$.
\end{lemma}

\begin{proof}
The eigenvalues of $N$  are 
\begin{align}
    \lambda_1 &=-\omega \\
    \lambda_2 &= \frac{1}{2}\left(\alpha \beta -\gamma - \delta - \sigma - \sqrt{(-\alpha \beta +\gamma + \delta + \sigma)^2 -4(-\alpha \beta \gamma +\delta\gamma - \beta \sigma + \gamma \sigma)}  \right) \\
    \lambda_3 &= \frac{1}{2}\left(\alpha \beta -\gamma - \delta - \sigma + \sqrt{(-\alpha \beta +\gamma + \delta + \sigma)^2 -4(-\alpha \beta \gamma +\delta\gamma - \beta \sigma + \gamma \sigma)}  \right).
\end{align}

It is not obvious, but algebra shows that all three eigenvalues are real since our parameters are positive: the discriminant simplifies to $4\beta\sigma+(\alpha \beta+\gamma-\delta-\sigma)^2$.  
Now $\lambda_1$ is clearly always negative.
Next, we have
\begin{align*}
    2\lambda_2 &= \alpha \beta -\gamma - \delta - \sigma - \sqrt{4\beta\sigma+(\alpha \beta+\gamma-\delta-\sigma)^2}\\
    &\leq \alpha \beta -\gamma - \delta - \sigma - (\alpha \beta+\gamma-\delta-\sigma)\\
    &=-2\gamma <0.
\end{align*}
Thus $\lambda_2$ is also always negative, and in fact we have
$$\lambda_2 \leq -\gamma.
$$
Finally, Mathematica\footnote{All Mathematica and Maple code used in this paper is available upon request.} shows that
$$\lambda_3 \begin{cases}
<0 \qquad \text{if} \quad  \mathfrak R_0<1  \\
=0 \qquad \text{if} \quad    \mathfrak R_0=1\\
>0 \qquad \text{if} \quad    \mathfrak R_0>1. 
 \end{cases}
$$
Thus $(n, 0, 0)$ is stable if and only if $ \mathfrak R_0<1$.  
\end{proof}


\subsection{SVE(R)IRS model}\label{sec: new SVE(R)IRS}
Since vaccines have not been available in previous pandemics, standard compartmental epidemiological models do not take vaccinated individuals into account. Consequently, controlling a pandemic had to be done solely via non-pharmaceutical mitigation measures. This changed during the COVID-19 pandemic, as effective vaccines were produced early on. Adding a vaccinated compartment to the model in Section \ref{sec: new SE(R)IRS} yields the following model, which we denote SVE(R)IRS:
$$
\xymatrix{V \ar@/^/[d]^\psi \ar@/^/[rrd]^{\rho \beta (I+\alpha E)/n}
 & & & & \\
S  \ar[rr]^{\beta (I+\alpha E)/n}   \ar@/^/[u]^\phi &  & E\ar[r]^{\sigma} \ar@/^2pc/[rr]^{\delta} &  I\ar[r]^{\gamma} & R \ar@/^1pc/[llll]^{\omega}
}
$$
Here $1-\rho$ represents the efficacy of the vaccine, $1/\psi$ is the duration of efficacy of the vaccine,  and $1/\phi$ is the rate at which people are vaccinated.  The first two are intrinsic to the vaccine itself, while $\phi$ can be thought of as a control (see Section \ref{sec: control}).  We may assume $\rho \in [0,1]$ and $\phi, \psi>0$.

The associated dynamics are given by:
\begin{align}
\frac{d S}{dt} &= -\beta S (I+\alpha E)/n + \omega R -\phi S + \psi V\\
\frac{d E}{dt} &=  \beta S (I+\alpha E)/n - (\sigma + \delta) E + \rho \beta V(I+\alpha E)/n \\
\frac{d I}{dt} &= \sigma E - \gamma I \\
\frac{d R}{dt} &= \delta E + \gamma I - \omega R \\
\frac{d V}{dt} &=- \rho \beta V(I+\alpha E)/n +  \phi S -  \psi V.
\end{align}

The dynamics of this model are significantly more complicated than those of the SE(R)IRS model in the previous section.
We prove the analogues of Lemmas \ref{lemma: new endemic relevance} and \ref{lemma: new disease free}.  We were unable to prove the analogue of Lemma \ref{lemma: new endemic stability}, although experimental evidence suggests that it does hold.

\begin{proposition} The SVE(R)IRS basic reproductive number is 
$$\mathfrak R_0 = \left(\frac{\beta}{\gamma} \right) \left(\frac{\alpha \gamma +\sigma}{\sigma + \delta}\right) \left(\frac{\psi+\rho \phi}{\psi + \phi} \right).$$
\end{proposition}

\begin{proof}
Again we follow \cite{Heffernan} by computing $\mathfrak R_0$ as the spectral radius of the next generation matrix.    We use tildes to not confuse with the compartment $V$.   First, a short computation shows that the system has a disease-free equilibrium at 
$p_1= (S, E, I, R, V)= \left(\frac{\psi n}{\phi+\psi}, 0, 0, 0,  \frac{\phi n}{\phi+\psi}\right).
$
We compute 
$$
\tilde F=\begin{pmatrix}
\frac{\alpha \beta}{\phi + \psi}(\psi + \rho \phi)  &  \frac{\beta}{\phi + \psi}(\psi + \rho \phi)\\ 0 & 0 
\end{pmatrix}
\qquad \text{and} \qquad 
\tilde V=\begin{pmatrix}
\sigma + \delta & 0 \\ -\sigma & \gamma 
\end{pmatrix}
$$
so our next generation matrix is 
$$
\tilde F \tilde V^{-1}=\begin{pmatrix}
\frac{\beta (\alpha \gamma +\sigma) (\psi + \rho \phi)}{\gamma(\psi + \phi)(\delta+\sigma)} &\frac{\beta  (\psi + \rho \phi)}{\gamma(\psi + \phi)}
 \\ 0 & 0 
\end{pmatrix}.
$$
The basic reproductive number is the spectral radius of this operator, which is the largest eigenvalue: $ \left(\frac{\beta}{\gamma} \right) \left(\frac{\alpha \gamma +\sigma}{\sigma + \delta}\right) \left(\frac{\psi+\rho \phi}{\psi + \phi} \right).$
\end{proof}

Here again, if we use the fact that $n=S+E+I+R+V$ is constant we can reduce the system and work with a 4-dimensional system in $S,E,I, V$.  The reduced equations are given by
\begin{align}
\frac{d S}{dt} &= -\beta S (I+\alpha E)/n + \omega (n-S-E-I-V) -\phi S + \psi V \label{eq: vax S dot}\\
\frac{d E}{dt} &=  \beta S (I+\alpha E)/n - (\sigma + \delta) E + \rho \beta V(I+\alpha E)/n \\
\frac{d I}{dt} &= \sigma E - \gamma I \\
\frac{d V}{dt} &=- \rho \beta V(I+\alpha E)/n +  \phi S -  \psi V. \label{eq: vax V dot}
\end{align}
In the sequel we will study this simpler version of the system, recovering the value of $R$ when convenient.

\subsubsection{Analysis of endemic equilibria} \label{subsec: endemic vax}


The SVE(R)IRS system has three equilibria.  One of them, denoted $p_1$, is the disease-free equilibrium that will be analyzed in Section \ref{subsec: new DFE vax}.  The other two, denoted $p_2, p_3$, correspond to endemic equilibria and are square-root conjugate to each other.  However, $p_2$ has negative components and is thus biologically irrelevant, while $p_3$ has all components positive and thus represents a true endemic equilibrium.  Formally, we have the following result, analogous to Lemma \ref{lemma: new endemic relevance}, whose proof appears in the \hyperref[appendix endeq]{Appendix}.

\begin{proposition}\label{prop: endeq vax}
If $\mathfrak R_0>1$, then there exists a unique endemic equilibrium point for the SVE(R)IRS system. 
\end{proposition}

The linearization of our system at $p_3$ is unwieldy; Mathematica cannot even determine when the determinant vanishes, let alone compute eigenvalues.  It can, however, produce the characteristic polynomial, so the methods of \cite{Fuller} applied in the proof of Lemma \ref{lemma: new endemic stability} could potentially work.  However, the characteristic polynomial is a complex expression and it is hard to tell whether the coefficients are positive; something similar is expected for $G$, the bialternate product of this matrix with itself.  
Therefore the stability of $p_3$ remains open.
For now we limit ourselves to examples, which provide hope that the analogues of Lemma \ref{lemma: new endemic stability} may indeed hold for the SVE(R)IRS model.

\begin{example}\label{ex: vax no endemic}
When 
$$(\alpha, \beta, \gamma, \delta, n, \sigma, \omega, \phi, \psi, \rho)=
\Big(\frac{1}{10}, \frac{1}{5}, \frac{1}{7}, \frac{1}{14}, 100, \frac{1}{7}, \frac{1}{90}, \frac{1}{360}, \frac{1}{180}, \frac{1}{10} \Big)
$$
we have $\mathfrak{R}_0 \approx 0.719$ and we
find that neither $p_2$ nor $p_3$ contains all positive coordinates.  So there is no relevant endemic equilibrium for these parameters.  Section \ref{subsec: new DFE vax} shows that there is in fact a stable disease-free equilibrium.

\end{example}

\begin{example}\label{ex: vax endemic}  Here we keep all parameters the same as in Example \ref{ex: vax no endemic} except $\beta$.  When 
$$(\alpha, \beta, \gamma, \delta, n, \sigma, \omega, \phi, \psi, \rho)=
\Big(\frac{1}{10}, \frac{9}{10}, \frac{1}{7}, \frac{1}{14}, 100, \frac{1}{7}, \frac{1}{90}, \frac{1}{360}, \frac{1}{180}, \frac{1}{10} \Big)
$$
we have 
$$ \mathfrak R_0 \approx 3.23$$
and a unique endemic equilibrium at 
$$p_3= (S,E,I,R,V) \approx (21,3,3,66,7).
$$

The eigenvalues of the reduced linearization at $p_3$ are
$$\lambda_1 \approx -.345, \quad \lambda_2 \approx -.009, \quad  \lambda_3 \approx -.020 + .053i, \quad  \lambda_4 \approx -.020 - .053i .$$
The four eigenvalues all have negative real parts, so the equilibrium is stable.  Note, however, that they are not all real, in contrast to the disease-free equilibrium case (see Proposition \ref{prop: new disease free vax} below).  We again see the appearance of a spiral sink due to epidemic waves (\cite{Bjornstad2}).  See Figure \ref{fig: vax endemic}.  
\begin{figure}[ht]
\centering
\includegraphics[width=7cm]{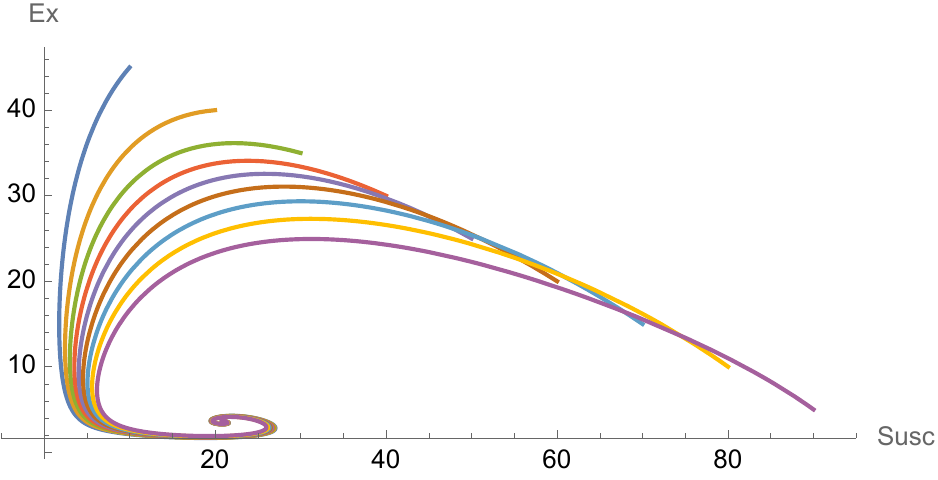}
\includegraphics[width=7cm]{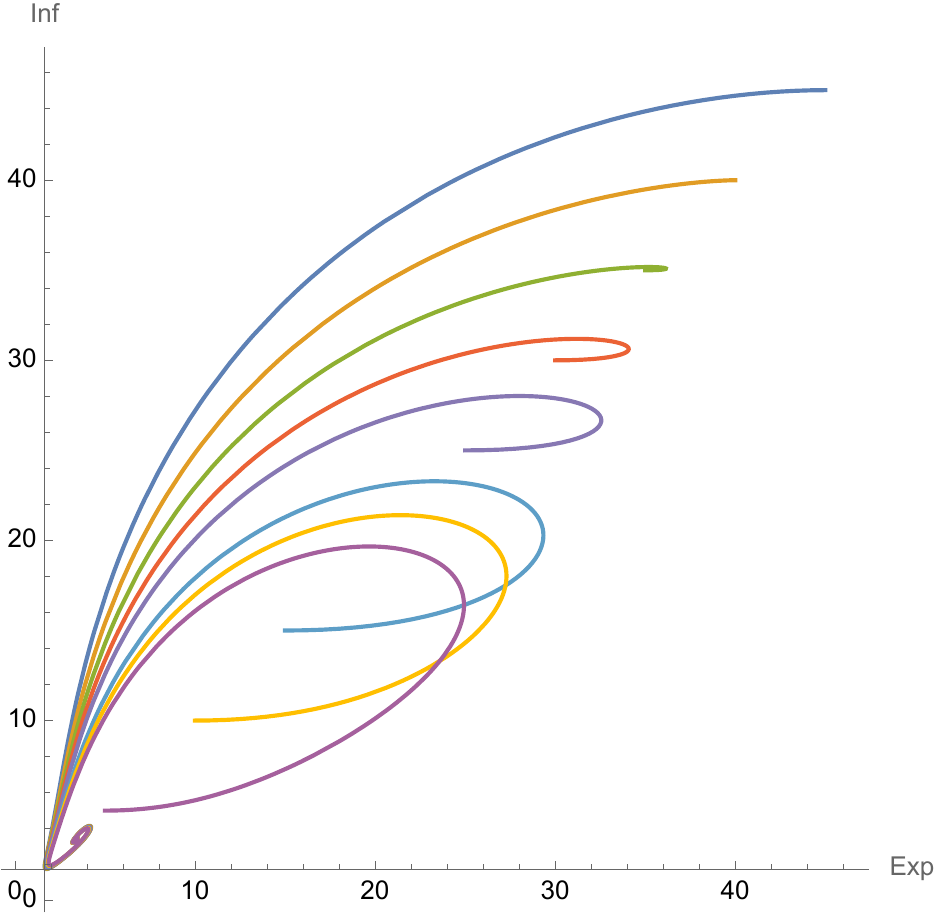} 
\caption{The endemic equilibrium of Example \ref{ex: vax endemic}. Plots as in Figure \ref{fig: endemic}.}\label{fig: vax endemic}
\end{figure}
\end{example}

\begin{example}\label{ex: varying beta}
Again we keep the parameter values as in Example \ref{ex: vax no endemic} but here we fix the initial condition $(S,E,I,R,V)=(90,5,5,0,0)$ and vary $\beta$ over time.  We begin with $\beta=0.5 (\mathfrak R_0 \approx 1.8)$ and run the dynamics until we approach the corresponding endemic equilibrium.  We then bump $\beta$ up to 0.75 and repeat, then finally bump $\beta$ up to 1.785 and repeat again.  See Figure \ref{fig: vax endemic varying beta}. This aligns with the observations in \cite{Chyba,Kunwar} and the real world data displayed in Figure \ref{fig:honolululoop}.
\begin{figure}[ht]
\centering
\includegraphics[width=7cm]{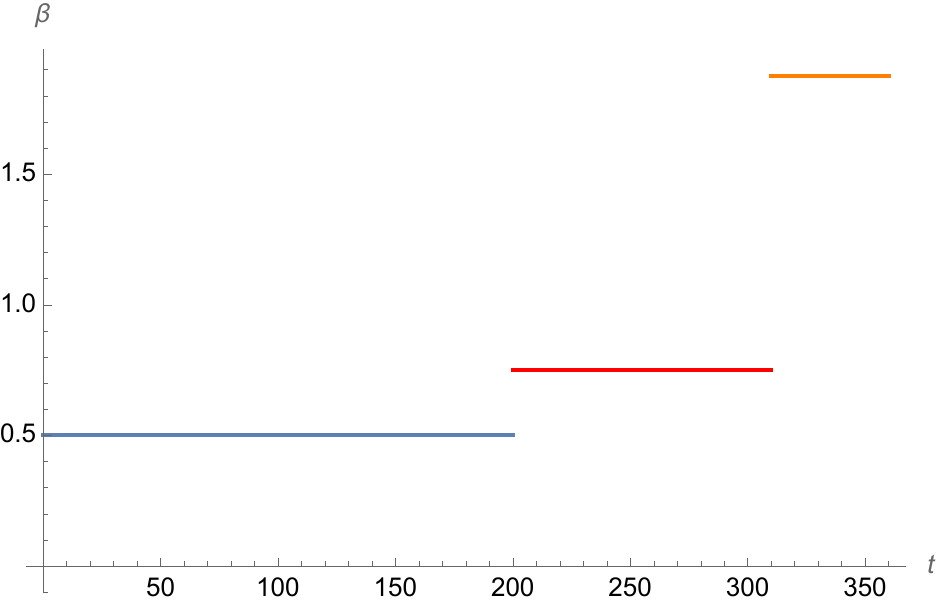}
\includegraphics[width=7cm]{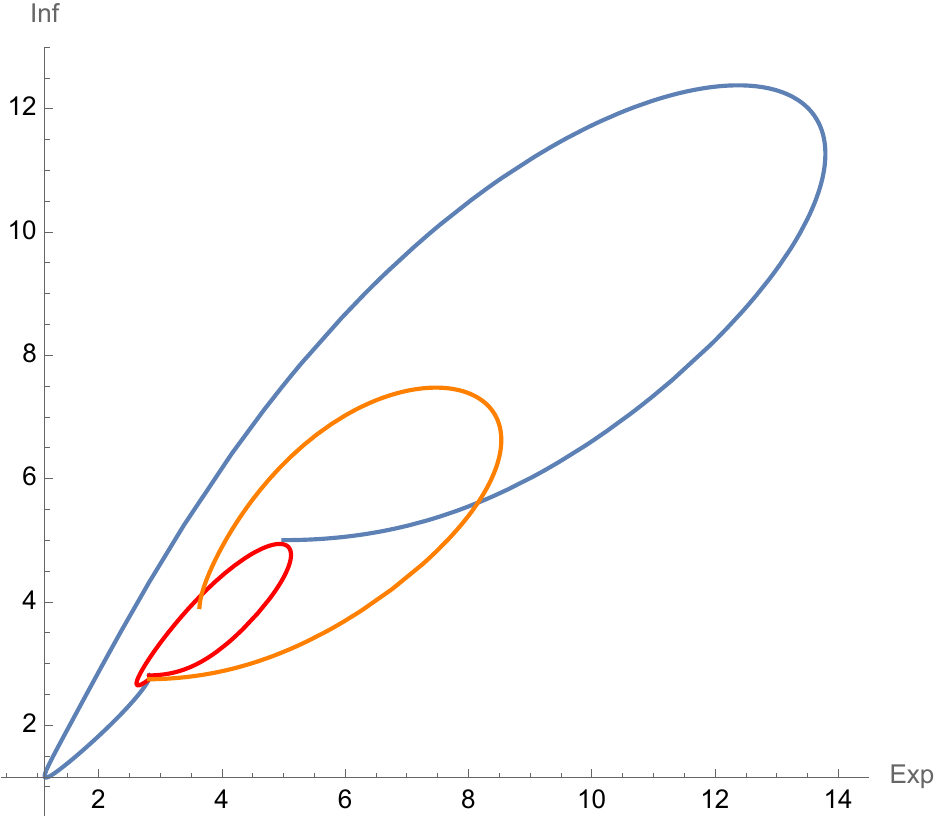}
\caption{{Left: Plot of sample $\beta$ varying over time.  Right: Plot of $E$ versus $I$ compartments with fixed initial condition and $\beta$ varying as in left plot.  Compare
with Figure \ref{fig:honolululoop}.} \\ 
}\label{fig: vax endemic varying beta}
\end{figure}

Figure \ref{fig:honolululoop} is obtained from the compartmental model presented in \cite{Chyba,Kunwar} and applied to the spread of COVID-19 in Honolulu County, Hawaii. The loops are due to changes in $\beta$ done through non pharmaceutical mitigation measures such as lockdown at the top and either relaxing mitigation measures or the appearance of new variants at the bottom. The black, blue, green and yellow curves are all aligned, and all correspond to the original strain of COVID-19. The red curve comes from the Delta surge that took place in Summer 2021, and the beginning of the Omicron surge can be seen in purple. The slope of the red curve differs from the one for the prior curves which is due to characteristics of the new variant and the fact that for a realistic system vaccinated individuals have a different probability to develop symptoms. Note that vaccination started in Hawaii in December 2020 which is during the green loop. For each variant the ratio of symptomatic versus asymptomatic was different which can be observed in the slope of the various loops.  
\begin{figure}[ht]
\centering
\includegraphics[width=16cm]{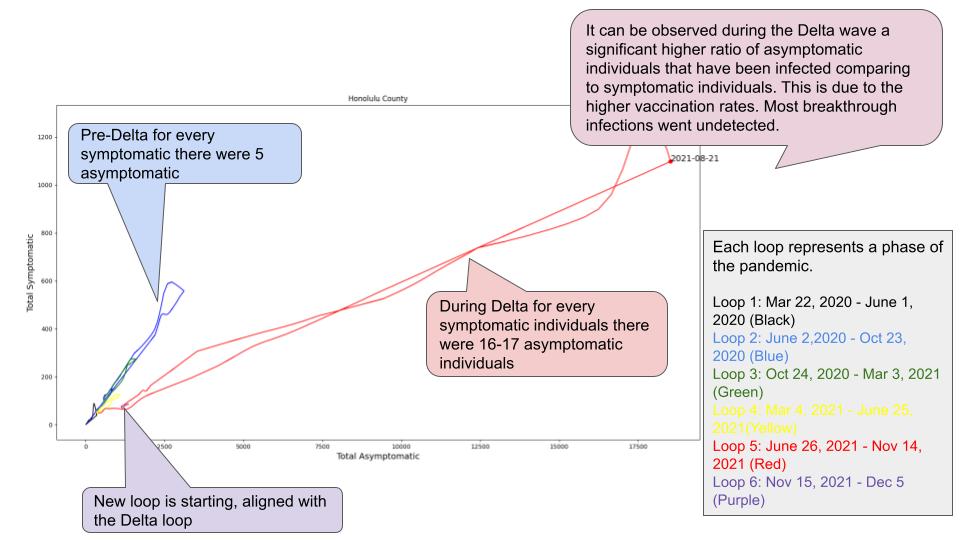}
\includegraphics[width=16cm]{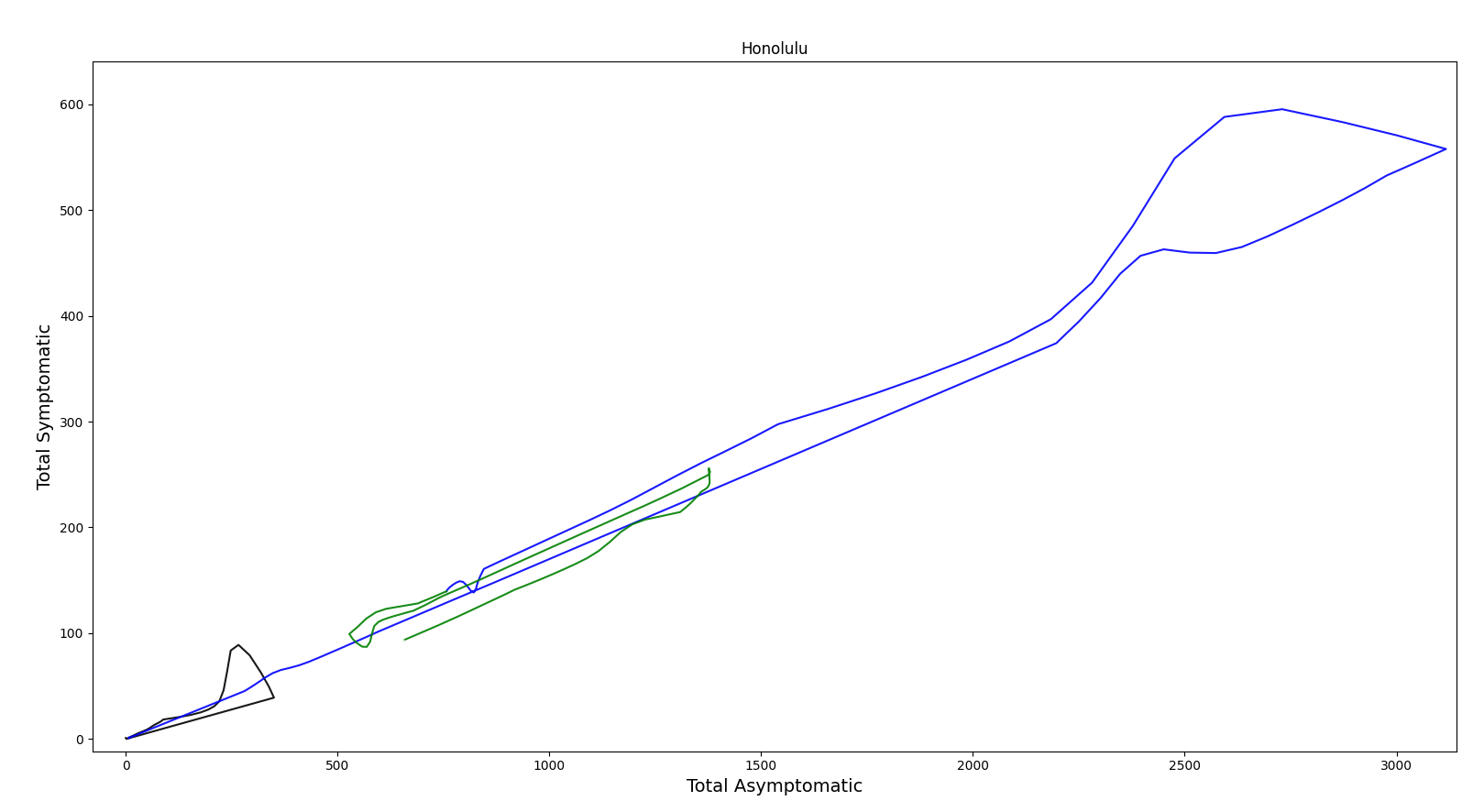}
\caption{I versus E compartments for Honolulu County throughout the COVID-19 pandemic. Bottom picture is a zoom on the period March 22, 2020 to March 3, 2021 without the Delta surge.}\label{fig:honolululoop}
\end{figure}
\end{example}

\bigskip

\subsubsection{Analysis of disease-free equilibrium}\label{subsec: new DFE vax}
Setting $E=I=0$ 
and solving the resulting system yields the unique disease-free equilibrium
$$ p_1=(S, E, I, R, V)= \left(\frac{\psi n}{\phi+\psi}, 0, 0, 0,  \frac{\phi n}{\phi+\psi}\right).
$$
Linearizing at this point yields the matrix
$$ N=
\begin{pmatrix}
 -\phi-\omega & \frac{\alpha \beta \psi}{\phi+\psi}- \omega & - \frac{\beta \psi}{\phi + \psi} -\omega & \psi-\omega \\
 0 & - \delta - \sigma - \frac{\alpha \beta (\psi+\rho \psi)}{\phi+\psi} & \frac{\beta (\psi + \rho\phi)}{\phi + \psi}  & 0 \\
 0 & \sigma & -\gamma & 0 \\
 \phi & -\frac{\alpha \beta \rho \phi}{\phi + \psi} & -\frac{ \beta \rho \phi}{\phi + \psi} & - \psi
\end{pmatrix}.
$$
As in the SE(R)IRS model, this matrix is singular if and only if $\mathfrak R_0 =1$.  

\begin{proposition}
\label{prop: new disease free vax}
The disease-free equilibrium is locally asymptotically  stable if and only if $\mathfrak R_0<1$.
\end{proposition}

\begin{proof}

The eigenvalues of $N$  are
\begin{align*}
    \lambda_1 &=-\omega \\
    \lambda_2 &=-\phi-\psi\\
    \lambda_3 &= \frac{1}{2(\phi + \psi)} \big(  \alpha \beta(\psi + \rho \phi) - (\phi + \psi)(\gamma+\delta+\sigma)
    - \sqrt{Z}\big)\\
     \lambda_4 &= \frac{1}{2(\phi + \psi)} \big(  \alpha \beta(\psi + \rho \phi) - (\phi + \psi)(\gamma+\delta+\sigma)
    + \sqrt{Z}\big)
\end{align*}
where large amounts of tedious algebra show that 
$$Z= 4\beta \sigma (\phi + \psi) (\psi + \rho \phi) + \big(\alpha \beta(\psi+ \rho\phi) + (\gamma-\sigma-\delta)(\phi+\psi)  \big)^2.
$$
This form of $Z$ makes it apparent that $Z$ is always positive and thus all four eigenvalues are always real. 

Now it is also clear that $\lambda_1$ and $\lambda_2$ are always negative.
Next, we have
\begin{align*}
    2(\phi+\psi)\lambda_3 &= \alpha \beta (\psi + \rho \phi)  - (\phi + \psi)(\gamma + \delta + \sigma) - \sqrt{Z}\\
 &\leq  \alpha \beta (\psi + \rho \phi)  - (\phi + \psi)(\gamma + \delta + \sigma) - \big(\alpha \beta(\psi+ \rho\phi) + (\gamma-\sigma-\delta)(\phi+\psi)  \big)\\
 &= -2\gamma(\phi + \psi)<0.
\end{align*}
Thus $\lambda_3$ is always negative as well, and in fact we have 
$$\lambda_3 \leq -\gamma.
$$

Finally, Mathematica shows that $\det N>0$ if and only if $\mathfrak R_0<1$. Since $\det N = \lambda_1 \lambda_2 \lambda_3 \lambda_4$ and $\lambda_1, \lambda_2, \lambda_3<0$, this shows that $\lambda_4<0$ if and only if $\mathfrak R_0<1$.  This proves the Proposition.
\end{proof}



\section{Control of the SVE(R)IRS system}\label{sec: control}

In this section we analyze the SVE(R)IRS model as an affine control system with two controls of the form:
\begin{equation}
    \dot {\hat q}(t)=F_0(\hat{q}(t))+u_1(t)F_1(\hat{q}(t))+u_2(t)F_2(\hat{q}(t)).
\end{equation}
The vaccine can be thought of as a first control over the system to steer the variables to desired values. The second control represents non-pharmaceutical interventions such as lockdown, social distancing, and mask mandates, and can also account for virus mutations (and therefore a change in transmission rate) that impact the parameter $\beta$. Introducing $\hat q=(S,E,I,R,V)^t$, we have:

\begin{equation}\label{SEIRV-Control-2}
\begin{aligned}
\dot {\hat q} &= \begin{pmatrix}
 \omega (n-S-E-I-V)  + \psi V\\
  - (\sigma + \delta) E \\
 \sigma E - \gamma I \\
  \delta E + \gamma I - \omega R \\
   -  \psi V
\end{pmatrix}+u_1\begin{pmatrix}
-S \\ 0 \\ 0 \\ 0 \\ S
\end{pmatrix}\\
&+u_2\begin{pmatrix}
- S (I +\alpha E)/n \\  (S + \rho V)(I +\alpha E)/n \\ 0 \\ 0 \\ - \rho V(I +\alpha E)/n
\end{pmatrix},
\end{aligned}
\end{equation}
where $u_1(t)=\phi(t)$ represents the vaccination rate and $u_2(t)=\beta(t)$ represents the transmission rate.

\subsection{Static feedback linearization}\label{subsec: SFL}
Given a control system, it may be possible to change the state and control variables in such a way that the new system is a linear control system. Historically, the first result in this direction is a theorem of Brunovsk\'y \cite{Brunovsky} which says that \text{all} controllable linear control systems may be put into a standard form, which is now famously known as \textit{Brunovsk\'y normal form}. In differential geometry language, such normal forms are essentially contact distributions of mixed order, otherwise known as generalized Goursat bundles \cite{VassiliouGoursat}\cite{VassiliouGoursatEfficient}. 

\begin{definition}\label{SFL}
A control system 
\begin{equation*}
\frac{d\textbf{x}}{dt}=\textbf{f}(\textbf{x},\textbf{u})
\end{equation*} 
is called \textbf{static feedback linearizable} (SFL) if there exists an invertible map $(t,\textbf{z},\textbf{v})=(t,\varphi(\textbf{x}),\psi(\textbf{x},\textbf{u}))$ such that the given control system transforms to a Brunovsk\'y normal form 
\begin{equation*}
\frac{d\textbf{z}}{dt}=A\textbf{z}+B\textbf{v}.
\end{equation*}
\end{definition}
The first results concerning when a given \textit{nonlinear} control system is SFL were given by Krener in \cite{KrenerLin}, as well as by Brockett in \cite{BrockettLin}, and then Jakubczyk and Respondek in \cite{RespondekLin}. Constructing explicit maps for SFL systems is harder, and that work was started by Hunt, Su, and Meyer in \cite{HuntSuMeyerLin}, and then a more geometric approach based on symmetry was developed in \cite{BrunovskySymmetry} by Gardner, Shadwick, and Wilkens, and finally work of Gardner and Shadwick \cite{GSalgorithm},\cite{GSFeedback}, and \cite{GSalgorithmExample} provided what is now known as the GS algorithm for static feedback linearization. The work of Vassiliou in \cite{VassiliouGoursat} and \cite{VassiliouGoursatEfficient} provides both a test for determining equivalence of generalized Goursat bundles and gives a procedure to construct appropriate diffeomorphisms.

The bi-input control system (\ref{SEIRV-Control-2}) 
fails to be controllable since the constant population requirement constrains trajectories to a hypersurface in the state-space; however, as mentioned before, we may reduce by the total population requirement. Indeed, this yields
\begin{equation}\label{SEIV-Control}
\begin{aligned}
\dot q &= \begin{pmatrix}
 \omega (n-S-E-I-V)  + \psi V\\
 - (\sigma + \delta) E  \\
 \sigma E - \gamma I \\
  -  \psi V
\end{pmatrix}
+u_1\begin{pmatrix}
-S \\ 0 \\ 0 \\ S
\end{pmatrix}\\
&+ u_2\begin{pmatrix}
-S (I +\alpha E)/n \\  (S + \rho V)(I +\alpha E)/n \\ 0 \\  - \rho V(I +\alpha E)/n
\end{pmatrix}
\end{aligned}
\end{equation}
where $q=(S,E,I,V)^t$. 

\begin{theorem}
\label{thm-SFL}
The control system (\ref{SEIV-Control}) is controllable and SFL with Brunovsk\'y normal form given by
\begin{equation*}
\dot{z}=A\,z+v_1\,b_1+v_2\,b_2,
\end{equation*}
where 
\begin{equation*}
    A=\begin{pmatrix} 0 & 1 & 0 & 0\\ 0 & 0 & 1 & 0\\ 0 & 0 & 0 & 0\\ 0 & 0 & 0 & 0\\ \end{pmatrix},\quad b_1=\begin{pmatrix} 0 \\ 0 \\ 1 \\ 0 \end{pmatrix}, \quad b_2=\begin{pmatrix} 0 \\ 0 \\ 0 \\ 1 \end{pmatrix},
\end{equation*}
and $z=(z_0,z_1,z_2,w_0)^t$ with the new controls $v_1$ and $v_2$. 
\end{theorem}

\begin{proof}
It is easily checked that the control system (\ref{SEIV-Control}) is bracket generating and hence controllable. While the authors implemented the procedure in \cite{VassiliouGoursatEfficient} via MAPLE to determine the linearizing map $\Phi$ below, one could determine the map by careful inspection after noting that the non-drift part of equation (\ref{SEIV-Control}) is linear in the state variables and control-affine in the control terms. It is straightforward to check this linearizing map by using the `Transformation', `Pushforward'/`Pullback' commands in the differential geometry package in MAPLE \cite{MAPLEDG}.
The feedback transformation $(z,v)=\Phi(q,u)$ is given by 
\begin{align*}
z_0&=n \left(  \left( S+E+I+V \right) \sigma+I\,\delta \right),\\
z_1&=-n\omega\,\sigma\,S-n\omega\,\sigma\,E- \left(  \left( 
\omega+\gamma \right) \sigma+\gamma\delta \right) nI\\
\,&+{n}^{2}\omega\,
\sigma-n\omega\,\sigma\,V,\\
z_2&=n{\omega}^{2}\sigma\,S+ \left( -\gamma{\sigma}^{2}+ \left( -\gamma\delta+
\delta\,\omega+{\omega}^{2} \right) \sigma \right) nE\\
\,&+ \left( 
 \left( \gamma^{2}+\gamma\omega+{\omega}^{2} \right) \sigma+\gamma^{2}\delta
 \right) nI-{n}^{2}{\omega}^{2}\sigma+n{\omega}^{2}\sigma\,V,\\
w_0&=S,\\
v_1&=\left( -\alpha\gamma{\sigma}^{2}+ \left( -\alpha\gamma\delta+\alpha\,\delta
\,\omega \right) \sigma \right) S\,E\,u_2\\
\,&+ \left( -\gamma{\sigma}^{2}+ \left( -\gamma\delta+\delta\,\omega \right) \sigma \right) u_2\,S\,I-n{\omega}^{3}\sigma\,S\\
\,&+ \left( -
\alpha\,\rho\gamma{\sigma}^{2}+ \left( -\alpha\gamma\delta\,\rho+\alpha\,
\delta\,\omega\,\rho \right) \sigma \right) E\,V\,u_2\\
\,&+ ( \gamma n{\sigma}^{3}+ \left( n\gamma^{2}+2\, \left( \delta+\omega/2
 \right) n\gamma-\delta\,\omega\,n \right) {\sigma}^{2}\\
 \,&+ \left( n\delta\,{\gamma
}^{2}+{\delta}^{2}n\gamma- \left( {\delta}^{2}+\delta\,\omega+{\omega}^{2}
 \right) n\omega \right) \sigma ) E\\
 \,&+ \left( -\gamma\rho\,{
\sigma}^{2}+ \left( -\gamma\delta\,\rho+\delta\,\omega\,\rho \right) \sigma
 \right) V\,I\,u_2\\
 \,&+ \left(  \left( -\gamma^{3}n-n\gamma^{
2}\omega-\gamma n{\omega}^{2}-n{\omega}^{3} \right) \sigma-\gamma^{3}\delta\,n
 \right) I\\
 \,&+{n}^{2}{\omega}^{3}\sigma-n{\omega}^{3}\sigma\,V,\\
v_2&=-\frac {u_2\,S(\alpha\,E+I)}{n}-(\omega+u_1)S-\omega(E+I)\\
\,&+ \left( -\omega+\psi \right) V+\omega
\,n.
\end{align*}
The inverse map is given by
\begin{equation}\label{Phi Inv}
\begin{aligned}
S(z)&=w_0,\\
E(z)&=\frac{-\gamma\omega z_0 -(\gamma+\omega)z_1-z_2+\sigma n^2 \omega\gamma}{nl},\\
I(z)&=\frac{-\omega z_0-z_1+n^2\omega\sigma}{l},\\
V(z)&=\frac{1}{nl}(((\omega+\gamma)\sigma+\gamma(\omega+\delta))z_0+(\sigma+\gamma+\delta+\omega)z_1+z_2,\\
\,&+n\sigma(\delta\omega-\gamma\delta-\gamma\sigma)w_0-n^2\sigma\omega(\sigma+\gamma+\delta)),\\
u_1(z,v)&=\frac{1}{nS(z)}(-u_2(z,v)S(z)(\alpha E(z)+I(z))+n\psi V(z)+n\omega R(z)-n v_2),\\
u_2(z,v)&=\frac{n{\omega}^{3}\sigma R(z)+C_{1,E}E(z)+
C_{1,I}I(z)-v_1
}{l(S(z)+\rho V(z))(\alpha E(z)+I(z))},
\end{aligned}
\end{equation}
where
\begin{align*}
l&=\sigma(\sigma\gamma+\delta(\gamma-\omega)),\\
C_{1,E}&=\sigma((\gamma-\omega)\delta+\gamma\sigma)(\omega+\delta+\gamma+\sigma)n,\\
C_{1,I}&=-((\delta+\sigma)\gamma^2+\gamma\omega\sigma+\omega^2\sigma)\gamma n,
\end{align*}
and
\[
R(z)=n-S(z)-E(z)-I(z)-V(z)
\]
is the recovered population in terms of the $z$ coordinates.
Notice that the controls $u_1$ and $u_2$ as functions of $z$ and $v$ are written via the $(S,E,I,V)$ state variables, and $u_1(z)$ is written also in terms of $u_2(z)$. Explicit expressions in terms of $z$ may be written, but are unnecessary. 
\end{proof}

The domain for the state variables $q$ is $[0,n]^4$; however, we may shrink to a smaller box region $\Omega_q$ to avoid the disease-free equilibrium. Since the map (\ref{Phi Inv}) is linear in the state variables $z$ and must have differential of full rank, then the domain $\Omega_q$ is mapped to a region $\Omega_z$ which is a parallelepiped in the $z$ coordinates. 

We now wish to understand the possible values for the controls $(v_1,v_2)$. In the $(S,E,I,V)$ space, the controls $(u_1,u_2)$ take values in the rectangle $\mathcal{U}=[0,\phi_{max}]\times[\beta_{min},\beta_{max}]$, where $0< \beta_{min}<\beta_{max}\leq 1$ and $\phi_{max} \leq 1$. Indeed, this yields Proposition \ref{v bounds prop} which is helpful for analyzing optimal control problems, as it classifies the possible $(v_1,v_2)$ regions by vertex behavior. First, we introduce the following expressions for convenience:
\begin{equation}\label{shorthand eqs}
\begin{aligned}
N_1(z)&=n{\omega}^{3}\sigma R(z)+C_{1,E}E(z)+
C_{1,I}I(z),\\
N_2(z)&=n\omega R(z)+n\psi V(z),\\
N_3(z)&=S(z)(\alpha E(z)+I(z)),\\
D_1(z)&=l(S(z)+\rho V(z))(\alpha E(z)+I(z)),\\
N_4(z)&=\frac{1}{n}\left(N_2(z)-\frac{N_1(z)N_3(z)}{D_1(z)}\right),\\
L(z,v_1)&=N_4(z)+\frac{N_3(z)}{nD_1(z)}v_1,
\end{aligned}
\end{equation}    
where it is clear that $N_2(z),N_3(z),D_1(z)>0$ on $\Omega_z$ and $L(z,v_1)$ has positive slope on $\Omega_z$ as a linear function of $v_1$.

The explicit static feedback transformation allows one to take trajectories of the linear control system to trajectories of the original control system. As such, the general problem of trajectory planning is eased by solving the planning problem for the linear control system first, then transforming the resulting trajectory to the original state-space. To this end -- and possible applications to optimal control in future work -- we present Proposition \ref{v bounds prop}, which describes the admissible controls for the linear problem. 

\begin{proposition}\label{v bounds prop}
Let
\begin{equation*}
\begin{aligned}
\mathcal{V}(z)&=\left\{(v_1,v_2)\in \mathbb{R}^2\colon f_{min}(z)\leq v_1 \leq f_{max}(z),\,g_{min}(z,v_1)\leq v_2 \leq g_{max}(z,v_1)\right\}
\end{aligned}
\end{equation*}
where
\begin{equation*}
\begin{aligned}
f_{min}(z)&= N_1(z)-\beta_{max}D_1(z),\\
f_{max}(z)&= N_1(z)-\beta_{min}D_1(z),\\
g_{min}(z,v_1)&= L(z,v_1)-\phi_{max}S(z),\\
g_{max}(z,v_1)&= L(z,v_1),
\end{aligned}
\end{equation*}
and $\Sigma_+\sqcup\Sigma_0\sqcup\Sigma_- = \Omega_z$ where
\begin{align*}
\Sigma_+&=\left\{z\in\Omega_z\colon \phi_{max}S(z)-\Delta\beta\frac{(\alpha E(z)+I(z))}{n}>0\right\},\\
\Sigma_0&=\left\{z\in\Omega_z\colon \phi_{max}S(z)-\Delta\beta\frac{(\alpha E(z)+I(z))}{n}=0\right\},\\
\Sigma_-&=\left\{z\in\Omega_z\colon \phi_{max}S(z)-\Delta\beta\frac{(\alpha E(z)+I(z))}{n}<0\right\},
\end{align*}
with $\Delta\beta=\beta_{max}-\beta_{min}$. Then $\mathcal{U}$ is mapped into one of three types of parallelograms in $(v_1,v_2)$ space: $\mathcal{P}_+, \mathcal{P}_0$, or $\mathcal{P}_-$ given by
\begin{align*}
\mathcal{P}_+&=\mathcal{V}(z)\vert_{\Sigma_+},\\
\mathcal{P}_0&=\mathcal{V}(z)\vert_{\Sigma_0},\\
\mathcal{P}_-&=\mathcal{V}(z)\vert_{\Sigma_-}.
\end{align*}
where `$\vert$' denotes set restriction.
\end{proposition}

\begin{proof}
First, under the backward static feedback map (\ref{Phi Inv}), the $u_1(z,v)$ and $u_2(z,v)$ may be written in terms of the expressions in (\ref{shorthand eqs}) as
\begin{align*}
u_1(z,v)&=\frac{1}{nS(z)}(-u_2(z,v)N_3(z)+N_2(z)-nv_2),\\
u_2(z,v)&=\frac{N_1(z)-v_1}{D_1(z)}.
\end{align*}
The region $\mathcal{U}$ in combination with the above equations immediately yields the inequalities
\begin{equation*}
N_1(z)-D_1(z)\beta_{max}\leq v_1 \leq N_1(z)-D_1(z)\beta_{min},
\end{equation*}
and 
\begin{equation}\label{v2 bounds 1}
L(z,v_1)-\phi_{max}S(z)\leq v_2\leq L(z,v_1),
\end{equation}
which is precisely the definition of the set $\mathcal{V}(z)$.
Since $N_2(z),N_3(z),$ and $D_1(z)$ are all positive on $\Omega_z$, we have that $L(z,v_1)$ has positive slope as a linear function in $v_1$ and therefore the region $\mathcal{V}(z)$ is a parallelogram for each $z\in\Omega_z$. Moreover, as $z$ changes continuously in $\Omega_z$, each parallelogram is continuously deformed to a new parallelogram with vertical edges and positively sloped $v_2$ bounds. The only major change in shape that is relevant for optimal control purposes is the position of the vertices of each parallelogram as functions of $z$. Let $\{(\zeta^i_z,\eta^i_z)\}_{i=1}^4$ be the vertices of each parallelogram labeled counter-clockwise starting from the highest value of $v_2$. Vertices $\eta^1_z$ and $\eta^3_z$ will remain the largest and smallest values of $v_2$ for each parallelogram for each $z$ because of the positive slope condition on $L(z,v_1)$; however, $\eta^2_z$ and $\eta^4_z$ are not fixed relative to each other as $z$ varies. Thus, there are three possible parallelogram types, $\eta^4_z>\eta^2_z$, $\eta^4_z=\eta^2_i$ and $\eta^4_z<\eta^2_z$. Writing out the expression $\eta^2_z-\eta^4_z$ gives
\begin{align*}
\eta^2_z-\eta^4_z&=L(z,\zeta^2_z)-L(z,\zeta^4_z)-\phi_{max}S(z),\\
&=\frac{N_3(z)}{nD_1(z)}(\zeta^2_z-\zeta^4_z)-\phi_{max}S(z),\\
&=\Delta\beta\frac{(\alpha E(z)+I(z))}{n}-\phi_{max}S(z).
\end{align*}
Thus, the case of $\eta^4_z-\eta^2_z=0$ corresponds exactly to a hyperplane $\Sigma_0\subset \Omega_z$ and $\Sigma_+$ and $\Sigma_-$ correspond to being on either side of this hyperplane such that $\eta^4_z-\eta^2_z$ is positive or negative, respectively. 
\end{proof}

It is worth explicitly stating that the inequalities defining the regions $\Sigma_{\pm}$ in Proposition \ref{v bounds prop} have an epidemiological interpretation. Namely, if the difference in transmissibility times the total proportion of infectious individuals is greater than, equal to, or less than the maximum possible vaccinated susceptible population at a given time, then the admissible control set is $\mathcal{P}_-$,$\mathcal{P}_0$, or $\mathcal{P}_+$ respectively. We conjecture that the switching times for the solutions of the time-optimal control problem (see subsection \ref{subsec: MP and singular}) for the original control system are determined by these inequalities. 

Finally, computations in Mathematica provide the following desirable result.
\begin{proposition}
The static feedback map $\Phi$ takes equilibria to equilibria.  That is, the disease-free equilibrium $p_1$ and the two endemic equilibria $p_2,\ p_3$ in the $S,E,I,R,V$ coordinates are all mapped to the equilibrium plane of the Goursat bundle, which consists of points of the form $(z_0, z_1, z_2, v_1, w_0, v_2)=(\text{constant}_1, 0, 0, 0, \text{constant}_2,0)$.
\end{proposition}

\begin{remark}
Alternatively, one can consider the single-input control system wherein the control is the vaccination rate only and the transmission rate is held fixed. In this case, the control system is also static feedback linearizable; however, despite the static feedback transformation being rational, the expressions are far more cumbersome. One can in principle recover this map from the bi-input system by fixing the control $u_2$ to be constant and then appending the resulting differential equation. The single-input SF map is available upon request. 
\end{remark}

\subsection{Simulations}\label{subsec: simulations}
Here we provide an example of control curves and plot the corresponding trajectories in the original variables.  We then transform both the states and controls according to the feedback transformation of Section \ref{subsec: SFL}.  Here we take $\phi=u_1$ and $\beta =u_2$ and all other parameters are as in Example \ref{ex: vax no endemic}.

\begin{figure}[H]
\centering
\includegraphics[width=5.5cm]{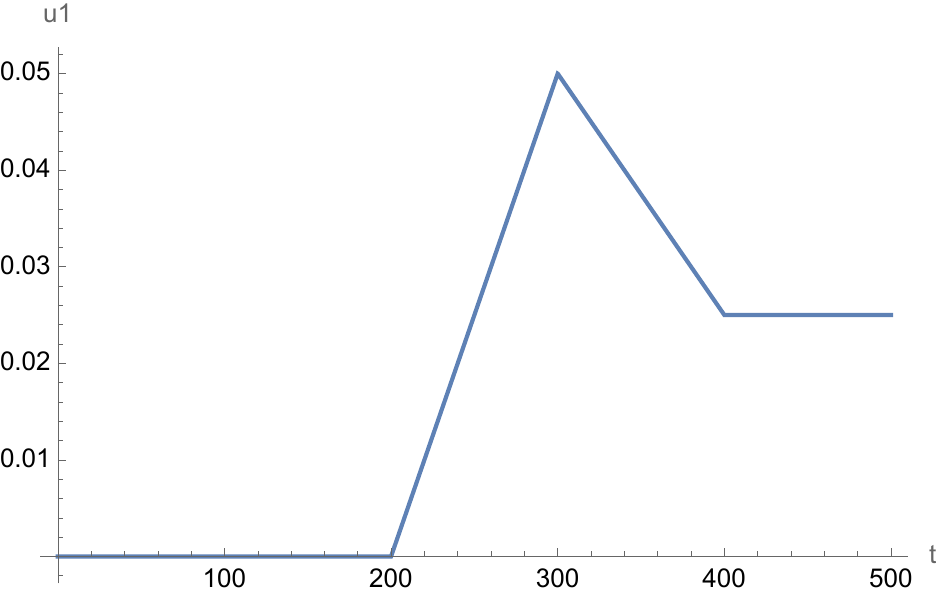}
\includegraphics[width=5.5cm]{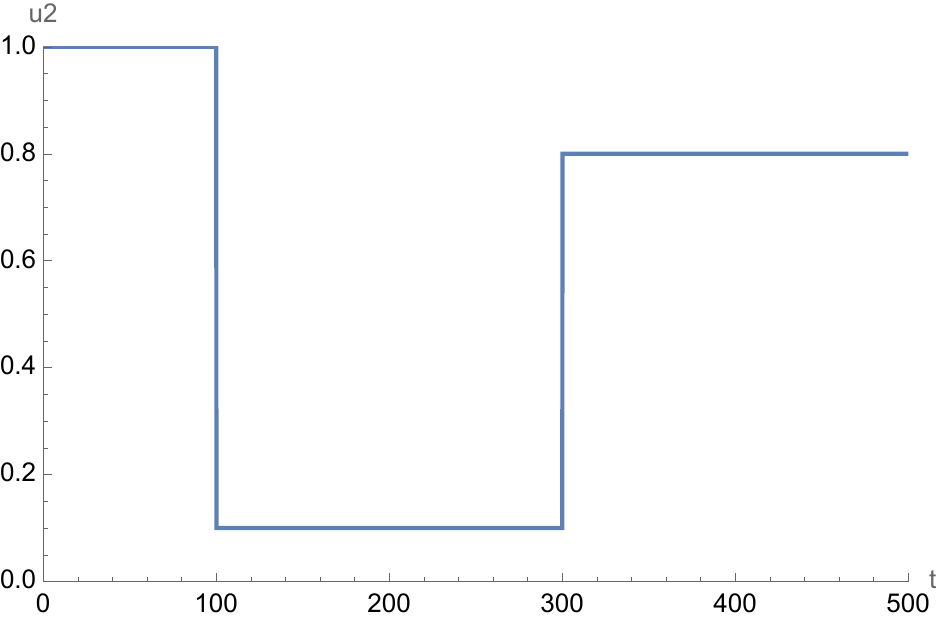}
\caption{Sample control curves $u_1(t)$ and $u_2(t)$.}
\label{fig: controls}
\end{figure}

In this hypothetical scenario, the epidemic begins with no vaccinations and a transmission rate of 1.  At time 100, a lockdown measure forces the transmission rate down to 0.1.  At time 200, vaccination begins and grows linearly, until time 300 when we have $u_1=0.05$.  At this time of peak vaccination, the lockdown ends, and the transmission rate jumps back up to 0.8 (lower than at the start of the epidemic, but much higher than during the lockdown) and stays there until the final time 500.  At time 300 the vaccination rate begins decreasing linearly until time 400, when it stabilizes at 0.025. See Figure \ref{fig: controls}.

The corresponding value of $\mathfrak R_0$ begins very high, then drops below 1 during the lockdown and decreases further during the vaccination campaign, but creeps back above 1 as the lockdown ends and vaccination rates stabilize.  See Figure \ref{fig: states}.  In this figure we also display the five state variables over time for this particular control scenario with the initial condition $(S, E, I, R, V)=(90, 5, 5, 0, 0)$.  

\begin{figure}[H]
\centering
\includegraphics[width=5.5cm]{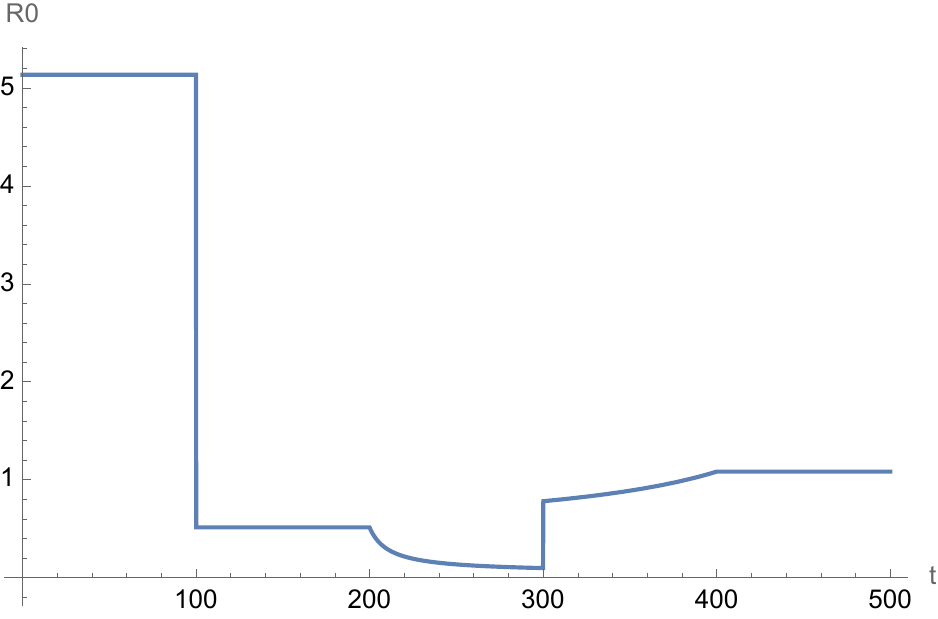}
\includegraphics[width=5.5cm]{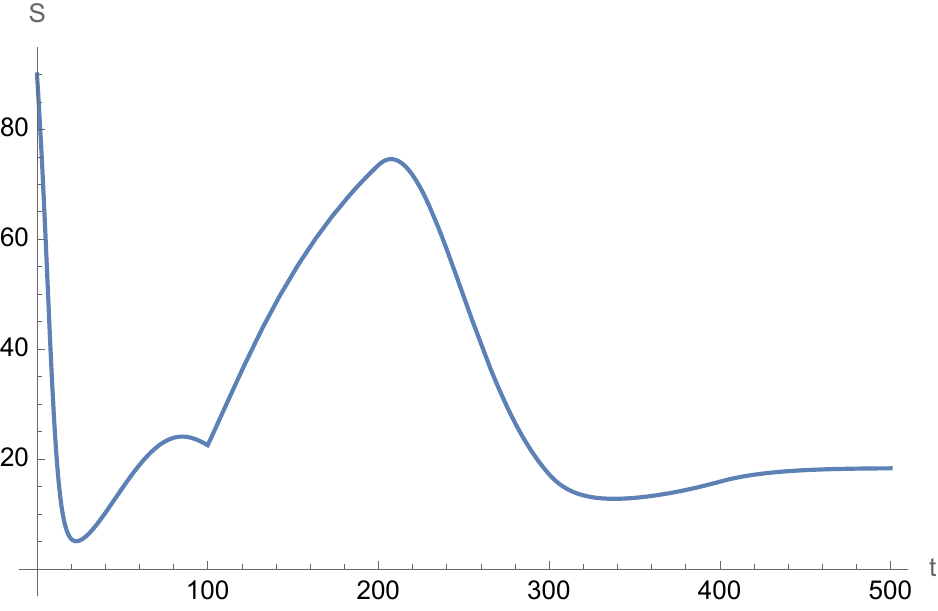}
\includegraphics[width=5.5cm]{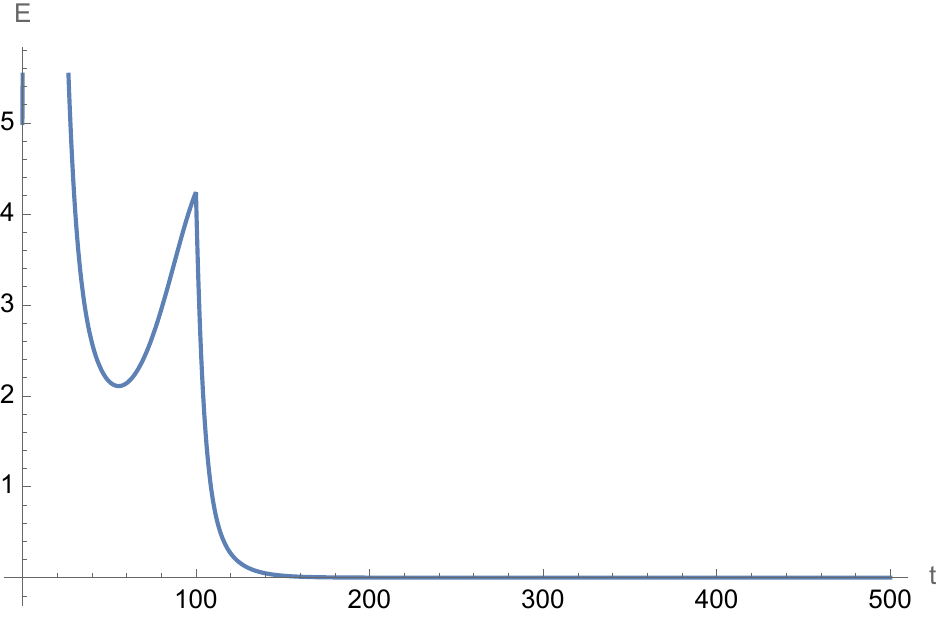}
\includegraphics[width=5.5cm]{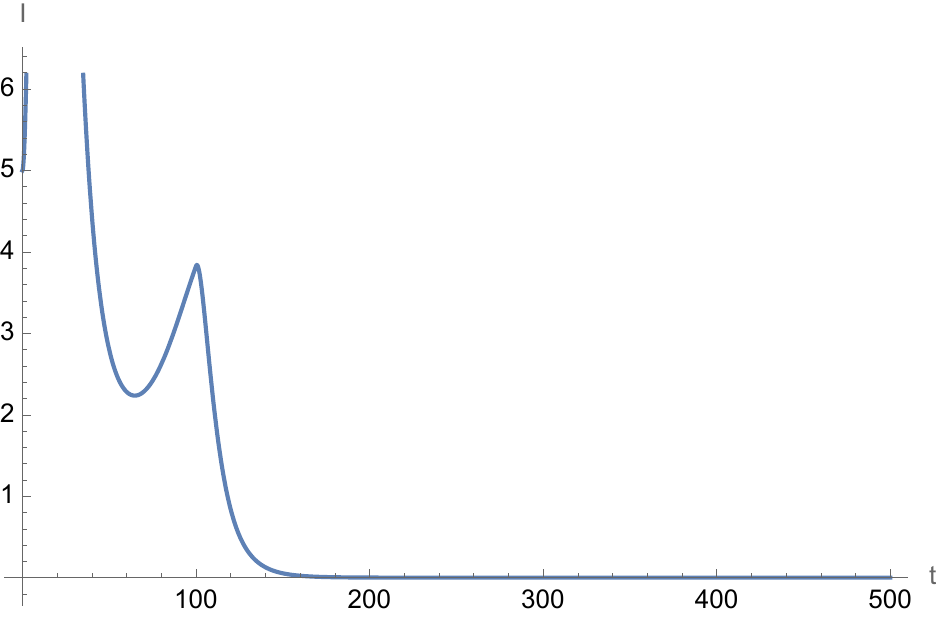}
\includegraphics[width=5.5cm]{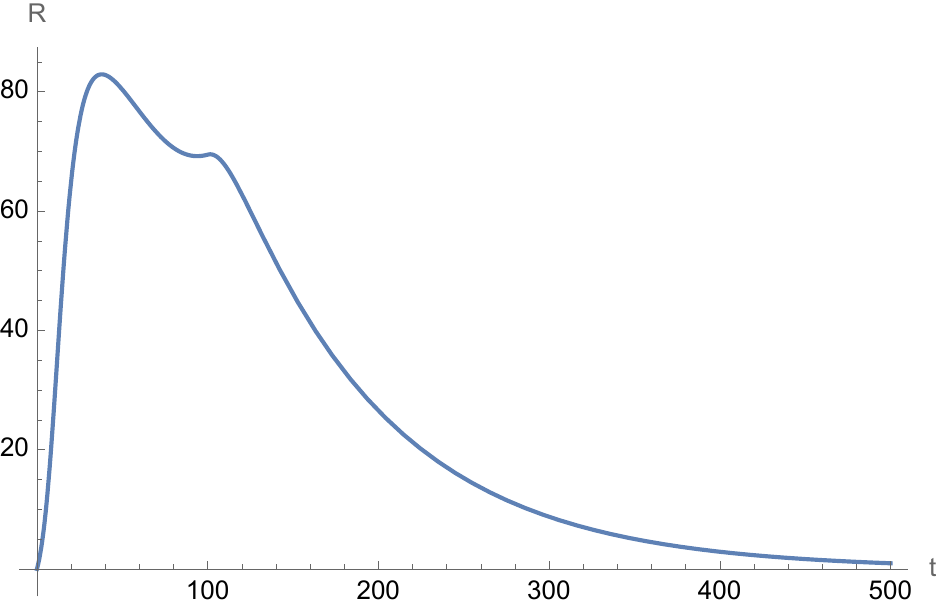}
\includegraphics[width=5.5cm]{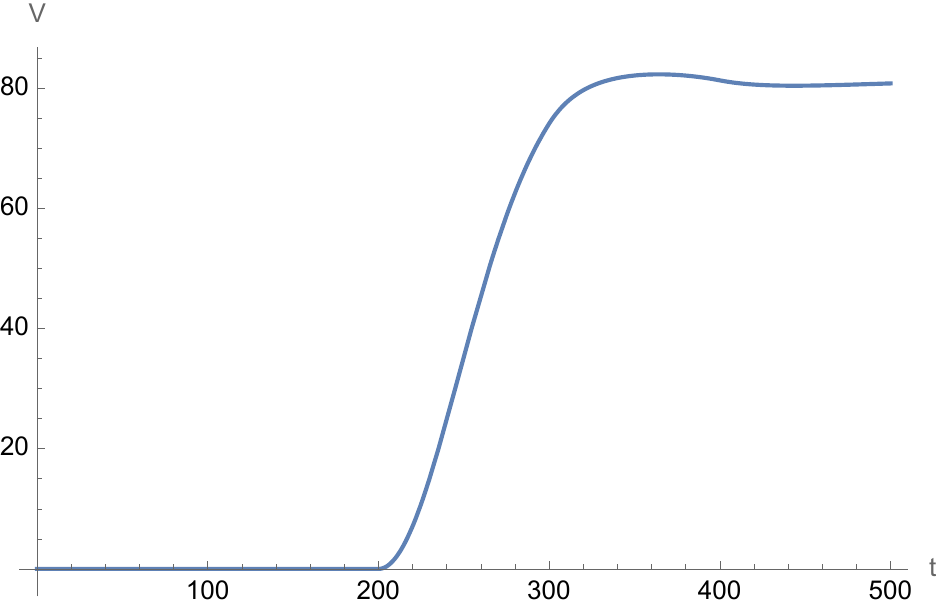}
\caption{The values of $\mathfrak R_0$ and all five states $(S, E, I, R, V)$ plotted over time, for the controls in Figure \ref{fig: controls}. The behavior is very similar to what happened for instance in the State of Hawaii in the first two years of the pandemic, the daily cases (here the I compartment) saw an initial surge followed by a rapid decrease and a new surge countered eventually by vaccination which resulted in plateauing provided that no new mutations took place \cite{Chyba}.}
\label{fig: states}
\end{figure}

In Figure \ref{fig: new states and controls} we display the same trajectory but in the new coordinates from the feedback transformation $\Phi$ from Section \ref{subsec: SFL}.  We show the states $z_0, z_1, z_2, w_0$ and the controls $v_1, v_2$ over time.   

\begin{figure}[H]
\centering
\includegraphics[width=5.5cm]{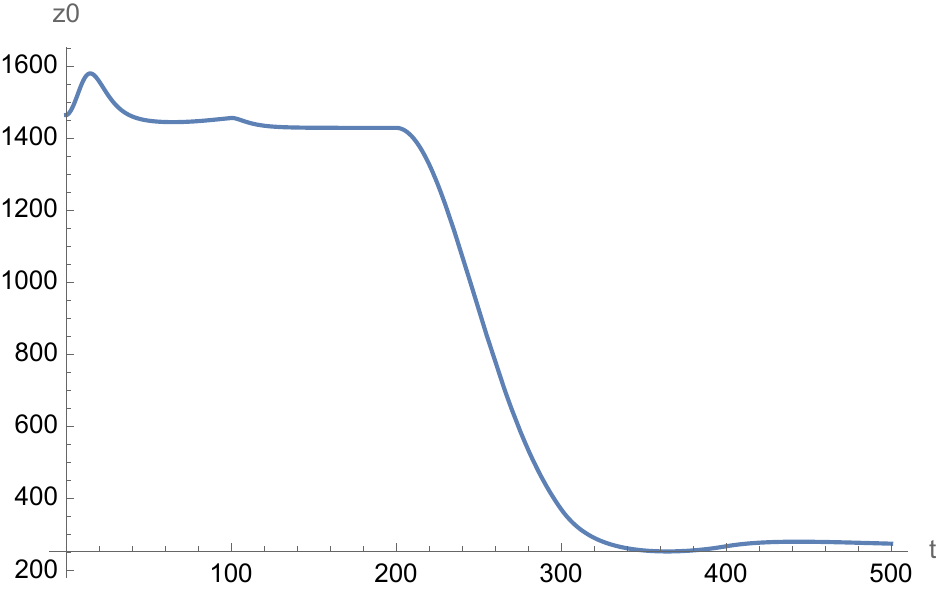}
\includegraphics[width=5.5cm]{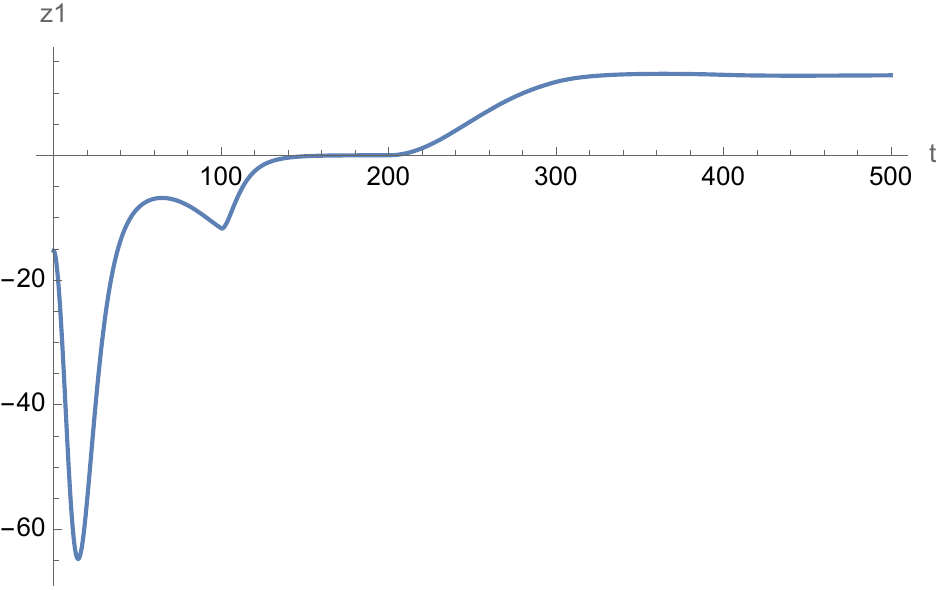}
\includegraphics[width=5.5cm]{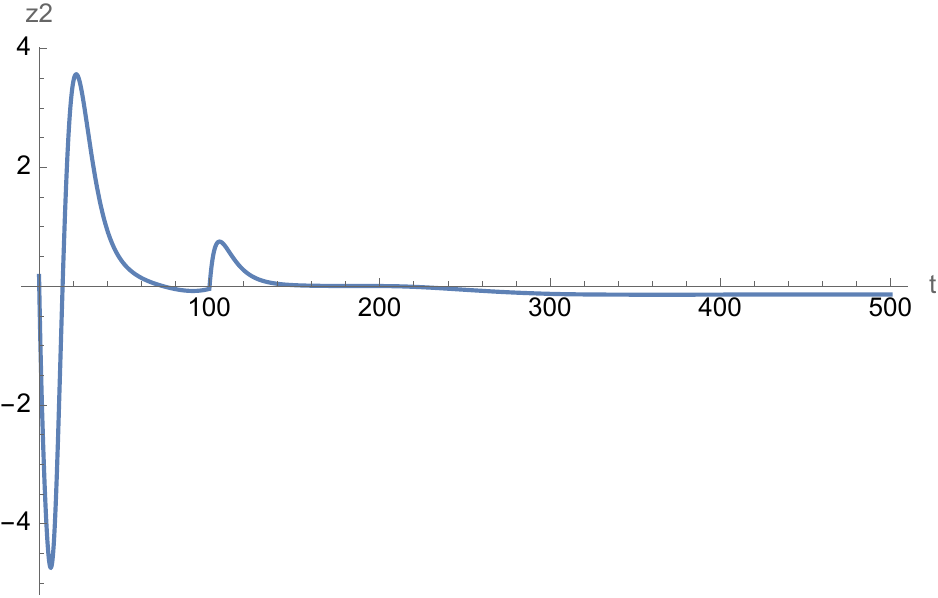}
\includegraphics[width=5.5cm]{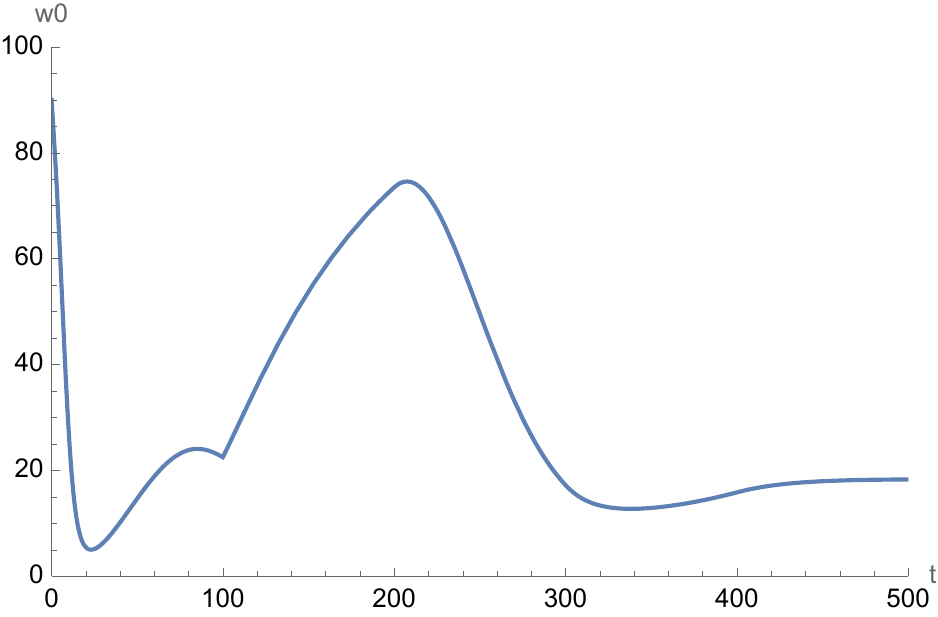}
\includegraphics[width=5.5cm]{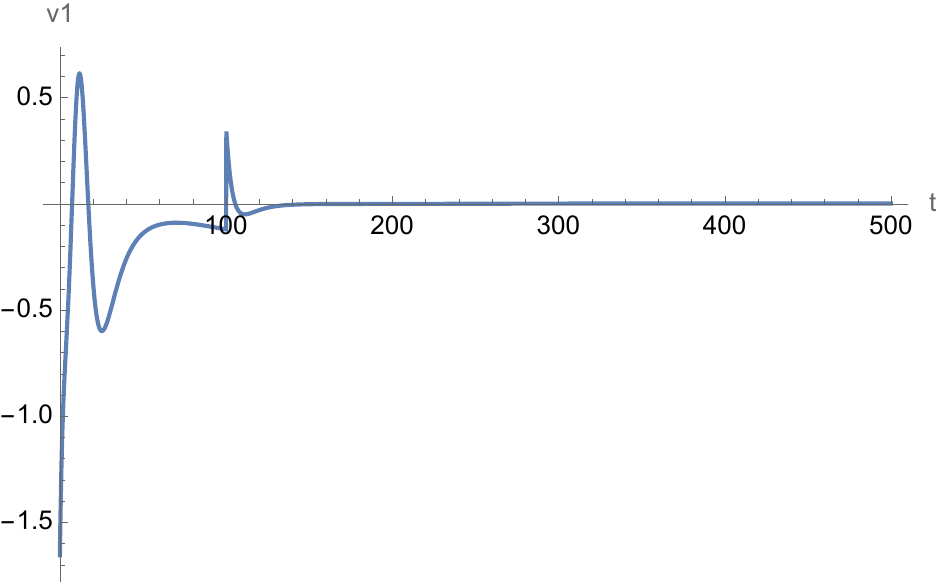}
\includegraphics[width=5.5cm]{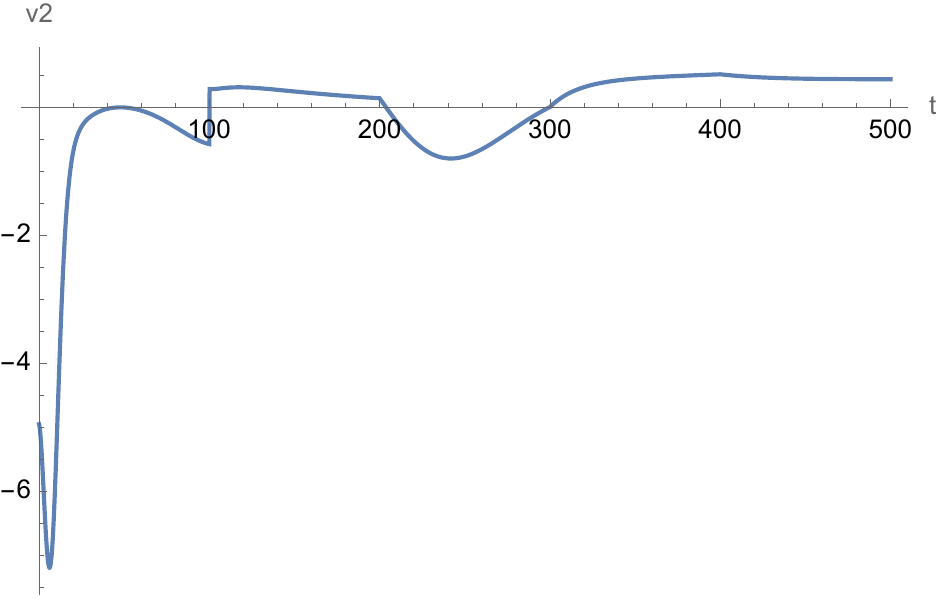}
\caption{The values of the four new states $z_0, z_1, z_2, w_0$ and two new controls $v_1, v_2$ over time, for the trajectories in Figures \ref{fig: controls} and \ref{fig: states}.}
\label{fig: new states and controls}
\end{figure}

Finally, in Figure \ref{fig: states} we can see that the state variables are trending toward the point $p \approx (17, 0.5, 0.5, 10, 71)$.  This point is the endemic equilibrium corresponding to the fixed parameters and final values of the controls $\beta = 0.8$ and $\phi = 0.025$.  In general this is expected: if the controls tend towards constant values then the states will tend toward the endemic equilibrium corresponding to these constant values (as long as $\mathfrak R_0>1$).  Thus in this scenario governments can control the shapes of the curves to potentially avoid overcrowding of hospitals or exhaustion of resources, but the long term trajectory of the epidemic depends only on the limiting values of the controls as $t \to \infty$.  In Table \ref{table} we display some values of the endemic equilibrium for particular choices of these values.  

\begin{table}[ht]
    \centering
    \begin{tabular}{c|c | c |c |c}
             & $\phi=0.0125$ & $\phi=0.025$ & $\phi=0.05$ & $\phi=0.1$\\ \hline
        $\beta=3$   & (6, 4, 4, 82, 4) & (6, 4, 4, 79, 8) & (5, 4, 4, 73, 14) & (4, 3, 3, 65, 24) \\ \hline
        $\beta=1.5$ & (12, 4, 4, 68, 13) & (11, 3, 3, 59, 25) & (8, 2, 2, 41, 46)& (5, 1, 1, 16, 77) \\ \hline
        $\beta=0.8$ & (21, 2, 2, 40, 35) & (17, .5, .5, 10, 71) & none & none
    \end{tabular}
    \caption{Approximate endemic equilibria $(S, E, I, R, V)$ for some values of $\beta$ and $\phi$.}
    \label{table}
\end{table}

\subsection{Optimal control}\label{subsec: optimal}
In this section we initiate a discussion of the optimal control problem of bringing the population to a desired state, such as decreasing the number of infections to an acceptable level as quickly as possible. We limit ourselves to showing the nonexistence of singular trajectories. This is an important result since, when they exist, singular extremals can play a very important role in the optimal synthesis. As an example, the sub-Riemannian Martinet distribution demonstrates that the sphere is not sub-analytic due to the existence of singular extremals \cite{Martinet}. A direct consequence of the nonexistence of singular extremals is that optimal controls must belong to the boundary of the control domain. 

To prove that there are no singular trajectories for the original bi-input system (\ref{SEIV-Control}), we use the fact that singular trajectories are feedback invariant and that they do not exist for the Brunovsk\'y normal form. The relationship between feedback classification and singular extremals was initially discussed in \cite{bonnard-1991}, we will here use the results presented in Chapter 4 of \cite{chyba-book}. 

In this paper, we consider an affine bi-input system of the form
\begin{equation}
\label{bi-inputpmp}
    \dot {q}(t)=F_0({q}(t))+u_1(t)F_1({q}(t))+u_2(t)F_2({q}(t))
\end{equation}
where $q(t)\in \mathbb R^4$. For our application the original vector fields $F_i$ are given in equation (\ref{SEIV-Control}) and as stated the control represents respectively the vaccination rate $u_1$ and the transmission rate $u_2$. The ones for the static feedback linear equivalent system can be found in Theorem \ref{thm-SFL} (given by $F_0=A$, $F_i=b_i, i=1,2$).
\ignore{
\begin{equation*}
\label{opt-singleinput}
    \dot {q}(t)=F_0({q}(t))+u_1(t)F_1({q}(t))
\end{equation*}
where
\begin{equation}\label{SEIV-Control-singleinput}
F_0(q) = \begin{pmatrix}
 \omega (n-S-E-I-V)-\beta S (I +\alpha E)/n  + \psi V\\
 - (\sigma + \delta) E +\beta (S + \rho V)(I +\alpha E)/n \\
 \sigma E - \gamma I \\
  -  \psi V -\beta\rho V(I +\alpha E)/n
\end{pmatrix}
\end{equation}
is the drift, the control vector field is given by 
\begin{equation*}
F_1(q) =\begin{pmatrix}
-S \\ 0 \\ 0 \\ S
\end{pmatrix}
\end{equation*}
and the control $u_1$ belongs to $[0,\phi_{\max}]$ where $\phi_{\max}$ represents the maximal vaccination rate that can be administered (bounded to supply and logistical constraints). }

We first define singular extremals, and then introduce the result that singular extremals are feedback invariants.
\subsubsection{Maximum Principle and Singular Extremals}\label{subsec: MP and singular}
A singular trajectory is defined as a singularity of the end-point mapping $E^{q_0,T}:u(.)\rightarrow q(q_0,T,u)$ where $q_0,T$ are fixed (\cite{chyba-book}). In particular, if $u$ is singular we have that $q(q_0,T,u)$ belong to the boundary of the accessibility set $\displaystyle{A(q_0,T)=\cup_{u(.)}q(q_0,T,u)}$. The following proposition provides a characterization of singular extremals using the weak maximum principle. 

Associated to system (\ref{bi-inputpmp}), we introduce the Hamiltonian function $H$ by:
\begin{equation}
    \displaystyle H(q,p,u)=\langle p,F_0(q)\rangle+\langle p,F_1(q)\rangle u_1+\langle p,F_2(q)\rangle u_2
\end{equation}
where $p:[0,T]\rightarrow \mathbb R^4$.  
\begin{proposition}[\cite{chyba-book}]
    Singular extremals $(q(t),p(t),u(t))$ are solutions of
    \begin{align*}
\label{Hamiltonian}
    \dot q(t)=\frac{\partial H}{\partial p}(q(t),p(t),u(t)),\qquad \dot p(t)&=-\frac{\partial H}{\partial q}(q(t),p(t),u(t)),\\
     \frac{\partial H}{\partial u}(q(t),p(t),u(t))&=0,
\end{align*}
almost everywhere, where $p:[0,T]\rightarrow \mathbb R^4$, $p\neq 0$ for all $t$ is called the adjoint vector.
\end{proposition}
The last equation is equivalent to stating that singular trajectories satisfy for all $t$ the constraint:
\begin{equation*}
\langle p(t), F_i(q(t))\rangle =0, \qquad i=1,2.
\end{equation*}
This constraint defines a subset $\Sigma_1$ of the space $\mathbb R^4\times \mathbb R^4\backslash \{0\}$. In \cite{bonnard-1991}, the author discusses the relation between singular extremals and the feedback classification problem for nonlinear systems; one the main results shows that singular extremals are feedback invariant. 
\begin{proposition}
For the two inputs SVE(R)IRS model (\ref{SEIV-Control}), singular extremals do not exist.   
\end{proposition}
\begin{proof}
    Using Theorem \ref{thm-SFL} and the linearizing map $\Phi$, we have shown that our system (\ref{SEIV-Control}) is feedback equivalent to the Brunovsk\'y normal form. The latter is a linear controllable system and therefore has no singular extremals. We conclude using Theorem 13, Chapter 4 in \cite{chyba-book}.
\end{proof}
Note that for the time optimal problem, this means that the control $u=(u_1,u_2)$ will take its value on the boundary of the domain of control $\mathcal{U}=[0,\phi_{max}]\times[\beta_{min},\beta_{max}]$.

Deriving the time-optimal synthesis is non-trivial due to the complexity of the image of the domain of control $\mathcal{U}$ under the map $\Phi$; see Proposition \ref{v bounds prop}. It could however be implemented numerically since the switching functions for the time-optimal problem associated to the Brunovsk\'y normal form with two inputs are simple. Moreover, it seems that other costs would be more interesting to study. 
\section{Discussion and Open Questions}\label{sec: discussion}
This paper touched upon two important aspects related to the evolution of a pandemic. First, the design and implementation of efficient strategies to control diseases spread in a population, which, as we saw during the COVID-19 pandemic, is of paramount importance. Second, understanding evolutionary scenarios toward a ``new normal", which is key for decision and policy makers to take proper actions at the correct time and avoid a prolonged stay in the pandemic phase. 

We discuss the above two aspects in more detail below, but we should also mention that this work introduced a large number of open questions as well as directions for generalization. The most obvious question is whether the endemic threshold property in Theorem \ref{thm: main thm} holds for the SVE(R)IRS model of Section \ref{sec: new SVE(R)IRS}.  
All the examples we have explored suggest that the answer to this questions is affirmative, but with so many free parameters an exhaustive search for counterexamples is challenging. A different set of open questions concerns generalizations of both our models and our results.  In particular, a natural addition to either model would be the vital dynamics of birth and death rates, as is common in the SEIRS literature.  
Other generalizations could include nonlinear transmission or seasonal forcing, as well as the appearance of new variants which would lead to having coefficients such as $\beta$ depend on time.  An alternative research direction would be the study of global rather than local stability of equilibria. Global stability for classical compartmental models has an extensive literature; see the comprehensive review article \cite{Hethcote}.  In particular, global stability was treated in \cite{Li2} and \cite{Liu} for SEIR, and in \cite{Cheng} and \cite{Li1} for SEIRS.  The methods of these and related papers might also prove effective for our models.  

\subsubsection*{Herd immunity versus endemic equilibrium}
In addition to our perceived need to consider compartmental models more closely adapted to COVID-19 than the classical models, this work was motivated in part by our observation that the national dialogue concerning the post-pandemic future focused largely on the concept of herd immunity rather than that of endemic equilibria.  According to \cite{Hethcote}, one has herd immunity when the immune fraction of the population exceeds $1-\frac{1}{\mathfrak R_0}$.  
In this case, the disease ``does not invade the population". However, for classical models as well as the models presented here, we know that as long as $\mathfrak R_0>1$ we still do not reach a disease-free equilibrium which would correspond to the complete eradication of the disease. To illustrate the shortcomings of the herd immunity concept, reconsider Example \ref{ex: vax endemic} which includes annual vaccinations.  Then we have  $1-1/\mathfrak R_0 \approx 0.69$.  According to \cite{Hethcote}, herd immunity thus occurs if more than 69\% of the populations is immune to the disease.  However, at the endemic equilibrium for these parameters, we have $R \approx 66$ and $V \approx 7$.  Thus 73\% of the population are immune, which is above the required threshold.  But over 6\% of the population are either in the $E$ or $I$ compartments ($\approx 3.4 \%$ each).  This shows that although we have technically achieved herd immunity, a large number of people still suffer from the disease. Ideally, sufficient vaccination could eradicate a disease completely.  By Proposition \ref{prop: new disease free vax}, this could happen if we took $\phi$ large enough to force $\mathfrak R_0<1$, as the dynamics would trend toward the disease-free equilibrium.  However, consider the following parameters, which are realistic for COVID-19:
$(\alpha, \beta, \gamma, \delta, n, \sigma, \omega, \psi, \rho)=
\Big(\frac{1}{10}, \frac{3}{10}, \frac{1}{7}, \frac{1}{14}, 100, \frac{1}{7}, \frac{1}{90}, \frac{1}{180}, \frac{1}{10} \Big)$. Here we have left $\phi$ free as a control.  We compute that $\mathfrak R_0 <1$ if and only if $\phi>1/282$.  Thus, in order to set a trajectory towards disease eradication, the population would need to be re-vaccinated more often than once per year, which seems unlikely as a long-term sustainable scenario.  Therefore for these parameters it seems more realistic that the best we could hope for is an endemic equilibrium with relatively small portion of the population infected at any given time.  With annual vaccinations we would have approximately $0.73\%$ of the population in the $E$ or $I$ compartments.

\subsubsection*{Efficient Control Strategies}
The control theoretic investigation initiated in Section \ref{sec: control} offers many potentially fruitful avenues for research. Proposition \ref{v bounds prop} provides a description of the corresponding domain of control in the Brunovsk\'y normal form, and can be used to analyze the switching functions for the static feedbcack linearization and therefore corresponding time-optimal controls. They can then be mapped back to the original systems via the inverse transformation.  More importantly, there are a large number of alternative cost functions from which one could reformulate the optimal control problem of Section \ref{subsec: optimal}: in particular, one may want to minimize the number of infections, or amount of vaccines administered, or the time taken to reach an endemic equilibrium, all subject to various boundary and endpoint conditions. The idea is to translate these cost functions via the SFL transformation given explicitly in the proof of Proposition \ref{thm-SFL}. The difficulty to solve these new optimal control problems expressed in the Brunovsk\'y normal form is to be determined, and it is expected that numerical techniques will come into play.

\subsection*{Acknowledgments}
The first and third authors were supported by NSF grant 2030789. The first three authors gratefully acknowledge support from the Coronavirus State Fiscal Recovery Funds via the Governor’s Office Hawaii Department of Defense. 
\section*{Appendix: Proof of Proposition \ref{prop: endeq vax}} \label{appendix endeq}
To find endemic equilibria of the SVE(R)IRS system, we need to solve the following system:
\begin{align*}
0 &= -\beta S (I+\alpha E)/n + \omega (n-S-E-I-V) -\phi S + \psi V \\
0 &=  \beta S (I+\alpha E)/n - (\sigma + \delta) E + \rho \beta V(I+\alpha E)/n \\
0 &= \sigma E - \gamma I \\
0 &=- \rho \beta V(I+\alpha E)/n +  \phi S -  \psi V.
\end{align*}
The third of the above equations readily yields $I=(\sigma/\gamma)E$. Plugging this into the rest of the equations and setting $\kappa = \frac{\sigma+\alpha\gamma}{n\gamma}$ we obtain
\begin{align*}
0 &= -\beta\kappa S E + \omega \left(n-S-\left(1+\frac{\sigma}{\gamma}\right)E-V\right) -\phi S + \psi V \\
0 &=  \beta \kappa S E - (\sigma + \delta) E + \rho \beta \kappa V E \\
0 &=- \rho \beta \kappa V E +  \phi S -  \psi V.
\end{align*}
Note that $E=0$ yields the disease-free equilibrium, so we may assume that $E\neq 0$. Adding the bottom two equations to the first one and dividing the second equation by $\beta\kappa E$ we get
\begin{align*}
0 &= \omega \left(n-S-\left(1+\frac{\sigma}{\gamma}\right)E-V\right) - (\sigma+\delta)E \\
0 &=  S - \frac{\sigma + \delta}{\beta\kappa} + \rho V \\
0 &=- \rho \beta \kappa V E +  \phi S -  \psi V.
\end{align*}
From the first two equations we easily find that
\begin{align*}
S &= -c + \rho a E \\
V &= b - a E,
\end{align*}
where
\begin{align*}
a &= \frac{1}{1-\rho}\left(1+\frac{\sigma}{\gamma} + \frac{\sigma + \delta}{\omega}\right) \\
b &= \frac{1}{1-\rho}\left(n - \frac{\sigma + \delta}{\beta\kappa}\right) \\
c &= \frac{1}{1-\rho}\left(\rho n - \frac{\sigma + \delta}{\beta\kappa}\right).
\end{align*}
Using the expressions for $S$ and $V$ in terms of $E$, we obtain the following quadratic equation:
\[
\rho\beta\kappa a E^2 +(\rho\phi a+\psi a - \rho\beta\kappa b) E - (\phi c + \psi b) = 0,
\]
which yields the following two roots
\begin{equation*}
E_{+,-} = \frac{1}{2\rho\beta\kappa a}\Big(
-(\rho\phi a+\psi a - \rho\beta\kappa b) \pm
\sqrt{(\rho\phi a+\psi a - \rho\beta\kappa b)^2+4\rho\beta\kappa a(\phi c + \psi b )}
\Big).
\end{equation*}
Note that
\begin{align*}
\phi c + \psi b &= \frac{\phi}{1-\rho}\left(\rho n - n\frac{\gamma(\sigma+\delta)}{\beta(\sigma+\alpha\gamma)}\right)+
\frac{\psi}{1-\rho}\left(n - n\frac{\gamma(\sigma+\delta)}{\beta(\sigma+\alpha\gamma)}\right)=\\
 &= \frac{n}{1-\rho}\left(\rho\phi+\psi-(\phi+\psi)\frac{\gamma(\sigma+\delta)}{\beta(\sigma+\alpha\gamma)}\right) =
 \frac{n(\rho\phi+\psi)}{1-\rho}\left(1-\frac{1}{\mathfrak{R}_0}\right).
\end{align*}
We thus see that if $\mathfrak{R}_0>1$ then both roots are real, and $E_{+}>0$ while $E_{-}<0$. The latter gives a biologically irrelevant state. Computing the values of $S$ and $V$ from $E_+$ we obtain
\begin{equation*}
V_{+} = b-a E_{+} = \frac{1}{2\rho\beta\kappa}\Big(\rho\beta\kappa b + \rho\phi a+\psi a - 
\sqrt{(\rho\phi a+\psi a - \rho\beta\kappa b)^2+4\rho\beta\kappa a(\phi c + \psi b )}
\Big),
\end{equation*}
\begin{multline*}
S_{+} = -c+\rho a E_{+} = \frac{1}{2\beta\kappa}\Big(-2\beta\kappa c -(\rho\phi a+\psi a -\rho\beta\kappa b) + \\
\sqrt{(\rho\phi a+\psi a - \rho\beta\kappa b)^2+4\rho\beta\kappa a(\phi c + \psi b )}
\Big).
\end{multline*}
Rather than showing directly from these expressions that $S_{+}$ and $V_{+}$ are positive, we proceed as follows. Note that instead of expressing $S$ and $V$ in terms of $E$ we could as easily express any two of these variables in terms of the remaining one. Hence, we can obtain quadratic equations for $S$ and $V$. We don't need the whole equations, but we note that both of them have a positive factor in front of the quadratic term, while the free constant term in the equation for $V$ is $\phi(\rho b -c)$, and in the equation for $S$ it is $-(\rho b -c)(\beta\kappa c/a+\psi)/\rho$. Now, $\rho b-c = \frac{\sigma+\delta}{\beta\kappa}>0$. Thus, the two roots of the equation for $V$ have the same sign. Since $E_{-}<0$ when $\mathfrak{R}_0>1$, we see that one of such roots is $b-a E_{-}>0$, hence we also have $V_{+}=b-a E_{+}>0$. Similarly, we note that if $c\geq 0$ the two roots of the equation for $S$ have opposite signs, and since one of them is $-c+\rho a E_{-}<0$, we have $S_{+}=-c+\rho a E_{+}>0$. The latter inequality is obvious when $c<0$.

Finally noting that $I_{+}=(\sigma/\gamma)E_{+}>0$, we see that having $\mathfrak{R}_0>1$ yields a single biologically relevant endemic equilibrium.

\end{document}